\documentclass[twoside,reqno]{amsart}
\usepackage{amssymb,amsmath,amsfonts,amsthm}
\usepackage{graphicx,graphics,psfrag,epsfig,calrsfs,enumerate,array}
\usepackage[colorlinks=true]{hyperref}
\hypersetup{urlcolor=blue, citecolor=green}

\newtheorem{theorem}{Theorem}
\newtheorem{remark}[theorem]{Remark}

\newtheorem{lemma}[theorem]{Lemma}

\newcommand{\Z}{{\mathbb Z}}
\newcommand{\R}{{\mathbb R}}

\newcommand{\T}{{\mathbb T}}

\def\duno{\partial_1}
\def\ddue{\partial_2}
\def\dtre{\partial_3}
\def\dh{\partial_h}

\def\dk{\partial_k}
\def\dbar{\overline\partial}
\def\dt{\partial_t}
\def\div{{\rm div\, }}

\def\eps{\varepsilon}
\def\F{\mathcal F}
\def\G{\mathcal G}
\def\H{\mathcal H}

\title[Incompressible current-vortex sheets]{A priori estimates for 3D incompressible\\ current-vortex sheets}
\author[j.-f. coulombel, a. morando, p. secchi and p. trebeschi]{}


\email{jean-francois.coulombel@math.univ-lille1.fr}
\email{alessandro.morando@ing.unibs.it}
\email{paolo.secchi@ing.unibs.it}
\email{paola.trebeschi@ing.unibs.it}

\thanks{The authors would like to warmly thank CIRM - FBK in Trento for its kind hospitality during the visiting 
period when this work was initiated. The first author was 
supported by the Agence Nationale de la Recherche, contract ANR-08-JCJC-0132-01. The last three authors were supported by the national research project PRIN 
2007 \lq\lq Equations of Fluid Dynamics of Hyperbolic Type and Conservation Laws\rq\rq. }

\date{\today}

\begin{document}
\maketitle
\centerline{\scshape Jean-Fran\c{c}ois Coulombel}
\smallskip
{\footnotesize
\centerline{CNRS, Universit\'e Lille 1 and Team Project SIMPAF of INRIA Lille Nord Europe}
\centerline{Laboratoire Paul Painlev\'e (UMR CNRS 8524), B\^atiment M2, Cit\'e Scientifique}
\centerline{59655 Villeneuve D'Ascq Cedex, France}
}
\medskip

\centerline{\scshape Alessandro Morando, Paolo Secchi, Paola Trebeschi}
\smallskip
{\footnotesize
\centerline{Dipartimento di Matematica, Facolt\`a di Ingegneria, Universit\`a di Brescia}
\centerline{Via Valotti, 9, 25133 Brescia, Italy}
}
\bigskip


\begin{abstract}
We consider the free boundary problem for current-vortex sheets in ideal incompressible magneto-hydrodynamics. 
It is known that current-vortex sheets may be at most weakly (neutrally) stable due to the existence of surface 
waves solutions to the linearized equations. The existence of such waves may yield a loss of derivatives in 
the energy estimate of the solution with respect to the source terms. However, under a suitable stability 
condition satisfied at each point of the initial discontinuity and a flatness condition on the initial front, 
we prove an a priori estimate in Sobolev spaces for smooth solutions with no loss of derivatives. The result 
of this paper gives some hope for proving the local existence of smooth current-vortex sheets without resorting 
to a Nash-Moser iteration. Such result would be a rigorous confirmation of the stabilizing effect of the magnetic 
field on Kelvin-Helmholtz instabilities, which is well known in astrophysics.
\end{abstract}

\section{Introduction}
\label{sect1}

\subsection{The Eulerian description}

We consider the equations of incompressible magneto-hydrodynamics (MHD), i.e. the equations governing the motion 
of a perfectly conducting inviscid incompressible plasma. In the case of a homogeneous plasma (the density $\rho 
\equiv$ const $>0$), the equations in a dimensionless form read:
\begin{equation}
\label{mhd1}
\begin{cases}
\partial_t u +\nabla \cdot 
(u \otimes u-H\otimes H) +\nabla q =0 \, ,\\
\partial_t H -\nabla \times
(u\times H) =0 \, ,\\
\div u=0\,, \; \div H=0 \, ,
\end{cases}
\end{equation}
where $u =(u_1,u_2,u_3)$ denotes the plasma velocity, $H=(H_1,H_2,H_3)$ is the magnetic field (in Alfv\'en velocity 
units), $q=p+|H|^2/2$ is the total pressure, $p$ being the pressure.

For smooth solutions, system \eqref{mhd1} can be written in equivalent form as
\begin{equation}
\label{euler}
\begin{cases}
\dt u +(u\cdot\nabla)u -(H\cdot\nabla)H  +\nabla q=0\, ,\\
\dt H +(u\cdot\nabla)H -(H\cdot\nabla)u  =0\, ,\\
\div u=0\,, \; \div H=0 \, .
\end{cases}
\end{equation}
We are interested in weak solutions of \eqref{mhd1} that are smooth on either side of a smooth hypersurface  $\Gamma(t)=\{x_3=f(t,x')\}$ in $[0,T]\times\Omega$, where $\Omega\subset\R^3,\, x'=(x_1,x_2)$ and that satisfy 
suitable jump conditions at each point of the front $\Gamma (t)$. For simplicity we assume that the density is 
the same constant on either side of $\Gamma(t)$.

Let us denote $\Omega^\pm(t)=\{x_3\gtrless f(t,x')\}$, where $\Omega=\Omega^+(t)\cup\Omega^-(t)\cup\Gamma(t)$; given any 
function $g$ we denote $g^\pm=g$ in $\Omega^\pm(t)$ and $[g]=g^+_{|\Gamma}-g^-_{|\Gamma}$ the jump across 
$\Gamma (t)$.

We look for smooth solutions $(u^\pm,H^\pm,q^\pm)$ of \eqref{euler} in $\Omega^\pm(t)$ such that $\Gamma (t)$ 
is a tangential discontinuity, namely the plasma does not flow through the discontinuity front and the magnetic 
field is tangent to $\Gamma (t)$, see e.g. \cite{landau}, so that the boundary conditions take the form
\begin{equation*}
\sigma =u^\pm\cdot n \, ,\quad H^\pm\cdot n=0 \, ,\quad [q]=0 \quad {\rm on } \;\Gamma (t) \, .
\end{equation*}
Here $n=n(t)$ denotes the outward unit normal on $\partial\Omega^-(t)$ and $\sigma$ denotes the velocity of 
propagation of the interface $\Gamma (t)$. With our parametrization of $\Gamma (t)$, an equivalent formulation 
of these jump conditions is
\begin{equation}
\label{RH}
\dt f =u^\pm \cdot N \, ,\quad H^\pm \cdot N=0 \, ,\quad [q]=0 \quad {\rm on }\;\Gamma (t) \, ,
\end{equation}
with $N:=(-\duno f,-\ddue f,1)$. Notice that the function $f$ describing the discontinuity front is part of 
the unknown of the problem, i.e. this is a free boundary problem.

System \eqref{euler}, \eqref{RH} is supplemented with initial conditions
\begin{equation}
\begin{array}{ll}\label{initial}
u^\pm(0,x)=u^\pm_0(x) \, ,\quad H^\pm(0,x)=H^\pm_0(x) \, ,\quad x\in \Omega^\pm(0) \, ,\\
f(0,x')=f_0(x') \, ,\quad x'\in \Gamma(0),
\end{array}
\end{equation}
where $\div u^\pm_0=\div H^\pm_0=0$ in $\Omega^\pm(0)$. The aim of this article is to show a priori estimates 
for smooth solutions to \eqref{euler}, \eqref{RH}, \eqref{initial}. This must be seen as a preliminary step 
before proving the existence and uniqueness of solutions to \eqref{euler}, \eqref{RH}, \eqref{initial}.

In the last years there has been a renewed interest for the analysis of free interface problems in fluid dynamics, especially for the Euler equations in vacuum and the water waves problem, see \cite{coutandshkoller1, coutandshkoller3} and the references thereinto. This fact has produced different methodologies for obtaining a priori estimates and the proof of existence of solutions. If the interface moves with the velocity of fluid particles, a natural approach consists in the introduction of Lagrangian coordinates, that reduces the original problem to a new one on a fixed domain. This approach has been recently employed with success in a series of papers by Coutand and Shkoller on the incompressible and compressible Euler equations in vacuum, see \cite{coutandshkoller1, coutandshkoller3}. However, this method seems hardly applicable to problem \eqref{euler}, \eqref{RH}, \eqref{initial}.

In the present paper we follow a different approach. To reduce our free boundary problem to the fixed domain, we consider a change of variables inspired from Lannes \cite{lannes}. The control of the function describing the free interface follows from a stability condition introduced by Trakhinin in \cite{trakhinin05}. The a priori estimate in Sobolev norm of the solution is then obtained by showing the boundedness of a higher-order energy functional.

\subsection{The reference domain $\Omega$}

To avoid using local coordinate charts necessary for arbitrary geometries, and for simplicity, we will 
assume that the space domain $\Omega$ occupied by the fluid is given by
$$
\Omega :=\{(x_1,x_2,x_3)\in\R^3 \; | \; x'=(x_1,x_2) \in \T^2 \, ,x_3\in(-1,1)\} \, ,
$$
where $\T^2$ denotes the $2$-torus, which can be thougt of as the unit square with periodic boundary conditions. 
This permits the use of {\it one} global Cartesian coordinates system. We also set
$$
\Omega^\pm := \Omega \cap \{ x_3\gtrless 0 \} \, ,\qquad \Gamma :=\Omega\cap\{x_3=0\} \, .
$$
On the {\it top} and {\it bottom} boundaries
$$
\Gamma_\pm := \{ (x',\pm 1) \, , \, x' \in \T^2 \}
$$
of the domain $\Omega$, we prescribe the usual boundary conditions
\begin{equation}
\label{bcfixed}
u_3=H_3=0 \quad {\rm on } \; [0,T] \times \Gamma_\pm \, .
\end{equation}
The moving discontinuity front is given by
$$
\Gamma(t) := \{ (x',x_3) \in \T^2 \times \R \, , \, x_3=f(t,x')\} \, ,
$$
where it is assumed that $-1<f(t,\cdot)<1$.

\subsection{An equivalent formulation in the fixed domain $\Omega$}

To reduce the free boundary problem \eqref{euler}, \eqref{RH}, \eqref{initial}, \eqref{bcfixed} to the fixed 
domain $\Omega$, we introduce a suitable change of variables that is inspired from \cite{lannes}. This choice 
is motivated below. In all what follows, $H^s(\omega)$ denotes the Sobolev space of order $s$ on a domain $\omega$. 
We recall that on the torus $\T^2$, $H^s(\T^2)$ can be defined by means of the Fourier coefficients and coincides 
with the set of distributions $u$ such that
\begin{equation*}
\sum_{n \in \Z^2} \big( 1+|n|^2 \big)^s \, |c_n(u)|^2 <+\infty \, ,
\end{equation*}
$c_n(u)$ denoting the $n$-th Fourier coefficient of $u$. The following Lemma shows how to lift functions from 
$\Gamma$ to $\Omega$.

\begin{lemma}[\cite{lannes}]
\label{lemma1}
Let $m \ge 1$ be an integer. Then there exists a continuous linear map $f\in H^{m-0.5}(\Gamma) \mapsto 
\psi \in H^m(\Omega)$ such that $\psi(x',0)=f(x')$ on $\Gamma$, $\psi(x',\pm 1)=0$ on $\Gamma_\pm$, and 
moreover $\dtre \psi (x',0)=0$ if $m\ge 2$.
\end{lemma}

\noindent For the sake of completeness, we recall the proof of Lemma \ref{lemma1} in Section \ref{prooflemma1} 
at the end of this article. The following Lemma gives the time-dependent version of Lemma \ref{lemma1}.

\begin{lemma}
\label{lemma2}
Let $m \ge1$ be an integer and let $T>0$. Then there exists a continuous linear map $f\in \cap_{j=0}^{m-1} 
{\mathcal C}^j([0,T];H^{m-j-0.5}(\T^2)) \mapsto \psi \in \cap_{j=0}^{m-1} {\mathcal C}^j([0,T];H^{m-j}(\Omega))$ 
such that $\psi(t,x',0)=f(t,x')$, $\psi(t,x',\pm 1)=0$, and moreover $\dtre \psi (t,x',0)=0$ if $m \ge2$. 
Furthermore, there exists a constant $C>0$ that is independent of $T$ and only depends on $m$, such that
\begin{multline*}
\forall \, f\in \cap_{j=0}^{m-1} {\mathcal C}^j([0,T];H^{m-j-0.5}(\T^2)) \, ,\quad 
\forall \, j=0,\dots,m-1 \, ,\quad \forall \, t \in [0,T] \, ,\\
\| \partial_t^j \psi (t,\cdot) \|_{H^{m-j}(\Omega)} \le C \, 
\| \partial_t^j f (t,\cdot) \|_{H^{m-j-0.5}(\T^2)} \, .
\end{multline*}
\end{lemma}

\noindent The proof of Lemma \ref{lemma2} is also postponed to Section \ref{prooflemma1}. The diffeomorphism 
that reduces the free boundary problem \eqref{euler}, \eqref{RH}, \eqref{initial}, \eqref{bcfixed} to the 
fixed domain $\Omega$ is given in the following Lemma.

\begin{lemma}
\label{lemma3}
Let $m \ge3$ be an integer. Then there exists a numerical constant $\eps_0>0$ such that for all $T>0$, 
for all $f \in \cap_{j=0}^{m-1} {\mathcal C}^j([0,T];H^{m-j-0.5}(\T^2))$ satisfying $\| f \|_{{\mathcal C} 
([0,T];H^{2.5}(\T^2))} \le \eps_0$, the function
\begin{equation}
\label{change}
\Psi(t,x) := \big( x', x_3 +\psi(t,x) \big) \, , \qquad (t,x) \in [0,T]\times \Omega \, ,
\end{equation}
with $\psi$ as in Lemma \ref{lemma2}, defines an $H^m$-diffeomorphism of $\Omega$ for all $t \in [0,T]$. 
Moreover, there holds $\partial^j_t \Psi \in {\mathcal C}([0,T];H^{m-j}(\Omega))$ for $j=0,\dots, m-1$, 
$\Psi(t,x',0)=(x',f(t,x'))$, $\Psi(t,x',\pm 1)=(x',\pm1)$, $\dtre \Psi(t,x',0)=(0,0,1)$, and
\begin{equation*}
\forall \, t \in [0,T] \, ,\quad \| \psi(t,\cdot) \|_{W^{1,\infty}(\Omega)} \le \dfrac{1}{2} \, .
\end{equation*}
\end{lemma}

\begin{proof}[Proof of Lemma \ref{lemma3}]
The proof follows directly from Lemma \ref{lemma2} and the Sobolev imbedding Theorem, because
$$
\dtre \Psi_3(t,x) =1 +\dtre \psi(t,x)\ge 1-\| \psi(t,\cdot) \|_{{\mathcal C}([0,T];W^{1,\infty}(\Omega))} 
\ge 1 -C \, \|f\|_{{\mathcal C}([0,T];H^{2.5}(\T^2))} \ge 1/2 \, ,
$$
provided that $f$ is taken sufficiently small in ${\mathcal C}([0,T];H^{2.5}(\T^2))$. In the latter inequality, 
$C$ denotes a numerical constant. The other properties of $\Psi$ follow directly from Lemma \ref{lemma2}.
\end{proof}

We set
\begin{equation*}
\begin{array}{ll}
A :=[D\Psi]^{-1} & ({\rm inverse\; of\; the\; Jacobian\; matrix}) \, ,\\
J :=\det\, [D\Psi] & ({\rm determinant\; of\; the\; Jacobian\; matrix}) \, ,\\
a :=J\, A & ({\rm transpose\; of\; the\; cofactor\; matrix}) \, ,
\end{array}
\end{equation*}
and we compute
\begin{equation}
\label{defAJa}
A =\begin{pmatrix}
1 & 0 & 0 \\
0 & 1 & 0 \\
-\partial_1\psi/J & -\partial_2\psi/J & 1/J \end{pmatrix} \, ,\qquad J =1 +\partial_3\psi \, ,\qquad 
a =\begin{pmatrix}
J & 0 & 0 \\
0 & J & 0 \\
-\partial_1\psi & -\partial_2\psi & 1 \end{pmatrix} \, .
\end{equation}
We already observe that under the smallness condition of Lemma \ref{lemma3}, all coordinates of $A$ are 
bounded by $2$ and $J \in [1/2;3/2]$. Now we may reduce the free boundary problem \eqref{euler}, \eqref{RH}, 
\eqref{initial}, \eqref{bcfixed} to a problem in the fixed domain $\Omega$ by the change of variables 
\eqref{change}. Let us set
$$
v^\pm(t,x) := u^\pm(t,\Psi(t,x)) \, ,\quad B^\pm(t,x) := H^\pm(t,\Psi(t,x)) \, ,\quad 
Q^\pm(t,x) := q^\pm(t,\Psi(t,x)) \, .
$$
Then system \eqref{euler}, \eqref{RH}, \eqref{initial}, \eqref{bcfixed} can be reformulated on the fixed 
reference domain $\Omega$ as
\begin{equation}
\label{mhd2}
\begin{cases}
\dt v^\pm +(\tilde v^\pm\cdot\nabla)v^\pm -(\tilde B^\pm\cdot\nabla)B^\pm  +A^T \, \nabla Q^\pm=0\, ,\\
\dt B^\pm +(\tilde v^\pm\cdot\nabla)B^\pm -(\tilde B^\pm\cdot\nabla)v^\pm  =0\, ,\\
(A^T \, \nabla) \cdot v^\pm =0\, ,\quad (A^T \, \nabla) \cdot B^\pm =0 \, , & 
{\rm in }\; [0,T] \times \Omega^\pm \, ,\\
\dt f = v^\pm \cdot N \, ,\quad B^\pm \cdot N=0 \, ,\quad [Q]=0 \, , & 
{\rm on }\; [0,T] \times \Gamma \, ,\\
v_3^\pm =B^\pm_3 =0 \, , & {\rm on }\; [0,T] \times \Gamma_\pm \, ,\\
v^\pm_{|t=0} =v^\pm_0 \, ,\quad B^\pm_{|t=0} =B^\pm_0 \, ,& {\rm on }\;  \Omega^\pm \, ,\\
f_{|t=0} =f_0 \, , & {\rm on }\; \Gamma \, .
\end{cases}
\end{equation}
In \eqref{mhd2}, we have set
\begin{equation}
\begin{array}{ll}
\label{defNtilde}
N :=(-\partial_1\psi,-\partial_2\psi,1) \, ,\\
\tilde v := A \, v -(0,0,  \dt \psi/J) =(v_1,v_2, (v\cdot N - \dt \psi)/J) \, ,\quad 
\tilde B := A \, B =(B_1,B_2, B\cdot N/J) \, .
\end{array}
\end{equation}
Vectors are written indifferently in rows or columns in order to simplify the redaction. Notice that
\begin{equation}
\label{bordo}
J =1 \, ,\qquad N=(-\duno f,-\ddue f,1)  \quad {\rm on }\; \Gamma \, ,\qquad 
\tilde v_3=\tilde B_3=0 \quad {\rm on }\; \Gamma \; {\rm and }\; \Gamma_\pm \, .
\end{equation}
We warn the reader that in \eqref{mhd2}, the notation $A^T$ is used to denote the transpose of $A$ and has 
nothing to do with the time interval $[0,T]$ on which the smooth solution is sought. We hope that this does 
not create any confusion.

\subsection{The main result}

\subsubsection{The linearized stability conditions}

The necessary and sufficient linear stability conditions for planar (constant coefficients) current-vortex 
sheets was found a long time ago by Syrovatskii \cite{syrovatskii} and Axford \cite{axford}. Let us consider 
constant vectors $u^\pm,H^\pm$ satisfying \eqref{RH} with the planar front $f(t,x') \equiv \sigma \, t +\xi' 
\cdot x'$ and constant pressures $q^\pm \equiv 0$. (Here we consider for this paragraph that $x'$ belongs to 
$\R^2$ instead of $\T^2$ and $x_3 \in \R$. This is however of no consequence on what follows.) The linear 
stability conditions for such piecewise constant solutions to \eqref{mhd1} read
\begin{subequations}
\label{syrovatskii}
\begin{align}
|[u]|^2 \le 2 \, \Big( |H^+| ^2+ |H^-|^2 \Big) \, ,\label{syrovatskiia}\\
|H^+\times [u]|^2 +|H^-\times [u]|^2 \le 2 \, |H^+\times H^-|^2 \, .\label{syrovatskiib}
\end{align}
\end{subequations}
Under the additional assumption $H^+ \times H^- \neq 0$, then \eqref{syrovatskiia} follows from 
\eqref{syrovatskiib} and the strict inequality in \eqref{syrovatskiia} follows from the strict inequality 
in \eqref{syrovatskiib}. The case of equality in \eqref{syrovatskiib} corresponds to the transition to 
{\it violent} instability, i.e. ill-posedness of the linearized problem. In the region of parameters 
defined by \eqref{syrovatskii}, the associated linearized equations admit surface waves of the form 
$\exp (i\, \tau \, t +i\, \eta \cdot x' -|\eta| \, |x_3|)$ for $\eta \in \R^2 \setminus \{ 0\}$ and 
some suitable $\tau \in \R$, see \cite{syrovatskii,axford} or \cite[page 510]{chandrasekhar}. We also 
refer to \cite{alihunter} for the derivation of weakly nonlinear surface waves.

The interior of the set of parameters described by \eqref{syrovatskii} is defined by the condition
\begin{equation}
\label{syrovatskii2}
|H^+\times [u]|^2 +|H^- \times [u]|^2 < 2\, |H^+\times H^-|^2 \, .
\end{equation}
In particular, $H^+\times H^- \neq 0$ and \eqref{syrovatskiia} becomes {\it redundant}. The condition 
\eqref{syrovatskii2} is always satisfied for {\it current} sheets, i.e. if $[u]=0$ and $H^+\times H^- 
\neq 0$. If $[u] \neq 0$, condition \eqref{syrovatskii2} can be rewritten as
\begin{equation*}
|[u]| <\dfrac{\sqrt{2} \, |H^+| \, |H^-| \, |\sin (\varphi^+-\varphi^-)|}{\sqrt{|H^+|^2 \, \sin^2\varphi^+ 
+|H^-|^2 \, \sin^2\varphi^-}} \, ,
\end{equation*}
where $\varphi^\pm$ denotes the oriented angle between $[u]$ and $H^\pm$.

Under the ``spectral stability condition'' \eqref{syrovatskii2}, Morando, Trakhinin and Trebeschi 
\cite{morandotrakhinintrebeschi} have shown an a priori estimate with a loss of three derivatives for 
solutions to the linearized equations with constant coefficients. In this paper we shall consider the 
following more restrictive situation:
\begin{equation}
\label{syrovatskii3}
\max \Big( |H^+\times [u]|,|H^- \times [u]| \Big) <|H^+\times H^-| \, .
\end{equation}
Under the latter more restrictive stability condition, which represents ``half'' of the stability domain 
defined by \eqref{syrovatskii2}, Trakhinin \cite{trakhinin052} has shown an a priori estimate in the 
anisotropic space $H^1_\ast$, without loss of derivatives from the data, for solutions of the linearized 
incompressible equations with variable coefficients. Similar stability conditions have also been considered 
by Trakhinin for the analysis of linearized and nonlinear stability of {\it compressible} current-vortex 
sheets, see \cite{trakhinin05,trakhinin09arma,ChenWang}. The choice of the space $H^1_\ast$ in 
\cite{trakhinin052} was motivated by the fact that the free boundary $\Gamma(t)$ is characteristic. 
However, we shall prove here that no loss of derivatives in the normal direction to the boundary occurs 
and we shall obtain estimates in standard Sobolev spaces. Though there is no loss of derivatives from 
the source terms of the equations to the solution in the main a priori estimate of \cite{trakhinin052}, 
the regularity assumptions on the coefficients are rather strong (stronger than what we shall assume 
here), and it is not so clear that the estimate in $H^1_\ast$ is sufficient to prove an estimate in 
some $H^m_\ast$, $m$ large enough, with coefficients in the same space $H^m_\ast$. There are even strong 
reasons to believe that with the formulation of \cite{trakhinin052}, a loss of regularity will occur with 
respect to the coefficients of the linearized equations.

Our goal here is to prove a {\it closed} estimate where coefficients are estimated in the same space 
as the data. As a matter of fact, we have found it more convenient to work directly on solutions to the 
nonlinear equations. Since we are considering classical solutions in three space dimensions, our a priori 
estimate will be proved in $H^3(\Omega)$, a space that is imbedded in $W^{1,\infty}$ by the Sobolev 
imbedding Theorem.

\subsubsection{The main result}

For a pair of functions $u =(u^+,u^-) \in H^s(\Omega^+)\times H^s(\Omega^-)$, with real $s\ge 1$, we will 
shortly write
\begin{equation*}
\| u^+ \|_{s,+} :=\| u^+ \|_{H^s(\Omega^+)} \, ,\quad \| u^- \|_{s,-} :=\| u^- \|_{H^s(\Omega^-)} \, ,\quad 
\| u^\pm \|_{s,\pm} := \| u^+ \|_{s,+} +\| u^- \|_{s,-} \, .
\end{equation*}
We also let $| \cdot |_{p,\pm}$ denote the $L^p$ norm on $\Omega^\pm$, and $| \cdot |_p$ denote the $L^p$ 
norm on $\Omega$ for $p\ge1$ and $p\not=2$; the $L^2$ norm on $\Omega^\pm$ is denoted by $\| \cdot ||_\pm$. Our main result reads as follows.

\begin{theorem}
\label{mainthm}
Let $\delta_0 \in \, ]0,1/2]$, let $R>0$, and let $v_0^\pm,B_0^\pm \in H^4(\Omega^\pm)$, $f_0 \in H^{4.5}(\T^2)$ 
satisfy
\begin{align}
\forall \, x' \in \T^2 \, ,\quad |B_0^+ \times B_0^- \, (x',0)| &\ge \delta_0 \, ,\notag \\
\forall \, x' \in \T^2 \, ,\quad \max \big( |B_0^+ \times [v_0] \, (x',0)|,|B_0^- \times [v_0] \, (x',0)| \big) 
&\le (1-\delta_0) \, |B_0^+ \times B_0^- \, (x',0)| \, ,\label{conditionthm} \\
\| v_0^\pm \|_{3,\pm} +\| B_0^\pm \|_{3,\pm} +\| f_0 \|_{H^{3.5}(\T^2)} &\le R \, .\notag
\end{align}
Then there exist $\eps_1>0$, $T_0>0$ and $C_1>0$ that depend only on $\delta_0$ and $R$ such that if 
$\| f_0 \|_{H^{2.5}(\T^2)} \le \eps_1$, then for all solution $(v^\pm,B^\pm,Q^\pm) \in {\mathcal C} 
([0,T];H^4(\Omega^\pm))$, $f \in {\mathcal C}([0,T];H^{4.5}(\T^2))$ to \eqref{mhd2} satisfying (without 
loss of generality)
\begin{equation*}
 \int_{\Omega^-} Q^-(t,x) \, {\rm d}x +\int_{\Omega^+} Q^+(t,x) \, {\rm d}x =0 \, ,
\end{equation*}
for all $t \in [0,T]$, the following estimates hold:
\begin{equation}
\begin{array}{ll}\label{finale}
\| v^\pm(t) \|_{3,\pm} +\| B^\pm(t) \|_{3,\pm} +\| Q^\pm(t) \|_{3,\pm} +&\| f(t) \|_{H^{3.5}(\T^2)} 
\le C_1 \, ,\\
&\| f(t) \|_{H^{2.5}(\T^2)} \le 2 \, \eps_1 \, ,
\end{array}
\end{equation}
for all $t \in [0,\min\{T,T_0\}]$.
\end{theorem}
Directly from \eqref{mhd2} and \eqref{finale} it readily follows a uniform estimate for $\|\partial_t  v^\pm(t) \|_{2,\pm}$, $\| \partial_t B^\pm(t) \|_{2,\pm}$ and $\| \partial_t f(t) \|_{H^{2.5}(\T^2)}$.

The first two conditions \eqref{conditionthm} are nothing but a uniform version of \eqref{syrovatskii3} on 
the initial front. Then our main result gives a uniform control of solutions to \eqref{mhd2} provided that 
a flatness condition is satisfied by the initial front. The main result also shows that the front remains 
sufficiently flat on a small time interval. The main interest of Theorem \ref{mainthm} is to show that 
energy estimates without loss of derivatives can be proved for \eqref{mhd2} in the framework of standard 
Sobolev spaces. We hope that in a near future, our approach will yield an existence and uniqueness result 
for \eqref{mhd2} without using a Nash-Moser iteration. As far as we know, no existence result has been 
proved yet for \eqref{mhd2}, with or without a Nash-Moser iteration.

\subsubsection{Strategy of the proof}

We consider the following energy functional
\begin{equation}
\label{energy}
{\mathcal E}(t) := \| v^\pm(t),B^\pm(t) \|^2_{3,\pm} +\| Q^\pm(t) \|_{3,\pm}^2 
+\| f(t) \|^2_{H^{3.5}(\T^2)} + \| \dt f(t) \|^2_{H^{2.5}(\T^2)} \, .
\end{equation}
Even though this function is not conserved, it is possible to show that $\sup_{t\in[0,T_0]}{\mathcal E}(t)$ 
remains uniformly bounded for sufficiently smooth solutions to \eqref{mhd2}, whenever $T_0>0$ is taken 
sufficiently small ($T_0$ being independent of the solution that we are considering). The strategy for 
proving Theorem \ref{mainthm} is the following: we first estimate the velocity and magnetic field by 
showing energy estimates on their tangential derivatives (meaning the $\duno$ and $\ddue$ derivatives), 
on their divergence and on their curl. Computing the curl equation is the crucial point if one wants to 
use standard Sobolev spaces (this is one difference with \cite{trakhinin052}). The front $f$ will be 
estimated directly from the boundary conditions in \eqref{mhd2}. Eventually, the pressure will be estimated 
by showing that $Q^\pm$ satisfy an elliptic system with source terms depending only on $v^\pm,B^\pm,f$ which 
have been estimated previously. Then we shall combine all these estimates to show that they yield a uniform 
control of solutions on a time interval that only depends on the size of the initial data.

Not so surprisingly, Theorem \ref{mainthm} requires an additional degree of regularity on the solution 
compared to the space in which we prove the estimate. This technical point is assumed only to justify 
all computations below (integration by parts and so on). This is exactly the same as when one proves a 
priori estimates for solutions to first order hyperbolic problems and in many aspects our analysis is 
closely linked to techniques used in hyperbolic boundary problems with characteristic boundaries. In 
particular, if we believe that coefficients of the differential operators in \eqref{mhd2} should have 
the same regularity as the solution to \eqref{mhd2}, then $A$ should belong to $H^3$ if $v^\pm,B^\pm$ 
belong to $H^3$. This forces the lifting $\psi$ of the front $f$ to belong to $H^4$ and this is where 
it is crucial to gain half-derivative from $f$ to $\psi$. This is the reason why we have adopted the 
same lifting procedure as in \cite{lannes}.

\section{Estimate of tangential derivatives}

\subsection{Uniform control of low order derivatives}

From now on we consider a time $T'>0$ such that we have for our given solution the uniform estimates:
\begin{subequations}
\label{uniform}
\begin{align}
\forall \, t \in [0,T'] \, ,\quad & \| f(t,\cdot) \|_{H^{2.5}(\T^2)} \le \eps_0 \, ,\label{uniforma} \\
&\| v^\pm(t) -v^\pm_0,B^\pm(t)-B^\pm_0 \|_{2,\pm} \le \eps_0 \, ,\label{uniformb}
\end{align}
\end{subequations}
where in \eqref{uniform}, the numerical constant $\eps_0$ is given by Lemma \ref{lemma3}. Let us 
already observe that with our choice of $\eps_0$, \eqref{uniforma} implies
\begin{equation*}
\forall \, (t,x) \in [0,T'] \times \Omega \, ,\quad |\nabla \psi (t,x)| \le \dfrac{1}{2} \, .
\end{equation*}
Moreover, the Sobolev imbedding Theorem implies that the $H^2$ norm dominates the $L^\infty$ norm 
on $\Omega^\pm$ so we can further restrict $\eps_0$, depending only on $\delta_0$, such that the 
following inequalities are implied by \eqref{uniformb}:
\begin{subequations}
\label{uniform'}
\begin{align}
\forall \, (t,x') \in [0,T'] \times \T^2 \, ,\quad &|B^+ \times B^- \, (t,x',0)| \ge 
\dfrac{\delta_0}{2} \, ,\label{uniformc} \\
\forall \, (t,x') \in [0,T'] \times \T^2 \, ,\quad 
&\dfrac{\max \big( |B^+ \times [v] \, (t,x',0)|,|B^- \times [v] \, (t,x',0)| \big)}
{|B^+ \times B^- \, (t,x',0)|} \le 1-\dfrac{\delta_0}{2} \, .\label{uniformd}
\end{align}
\end{subequations}
Of course, the time $T'$ chosen above a priori depends on the particular solution that we are considering, 
and one of our goals is to show below that $T'$ can be chosen to depend only on $\delta_0$ and on the norm 
$R$ of the initial data.

We will denote generic numerical constants (for instance constants that appear in Sobolev imbeddings) 
by the same letter $C$ or by $M_0$. Such constants are allowed to depend only on $\delta_0$ and $R$. 
We also let $F$ denote a generic nonnegative nondecreasing function which does not depend on the solution. 
In particular, we feel free to use $F+F=F$, $F \times F =F$ and so on. We shall sometimes write $u(t)$ 
instead of $u(t,\cdot)$, for some given function $u$ depending on $t$ and $x$. For shortness we shall 
write $\|v^\pm,B^\pm\|_{3,\pm}$ for $\|v^\pm\|_{3,\pm}+\|B^\pm\|_{3,\pm}$, and similarly for 
$\|\dt v^\pm,\dt B^\pm\|_{2,\pm}$ and other quantities. Let us now turn to the derivation of 
$L^2$ estimates for tangential derivatives of the velocity and magnetic field.

\subsection{Estimates of tangential derivatives}

Let us denote by $\overline\partial=(\partial_1,\partial_2)$ the horizontal (tangential) derivatives. Inspired from \cite{trakhinin05,trakhinin052} we define on $[0,T]$ the energy functional
\begin{equation}
\label{energy2}
\displaystyle \H(t) :=\dfrac1{2} \, \sum_\pm \sum_{|\alpha| \le 3} \int_{\Omega^\pm} 
\begin{pmatrix}
1 & -\lambda^\pm\\
-\lambda^\pm & 1 \end{pmatrix} \, \begin{pmatrix}
\dbar^\alpha v^\pm \\
\dbar^\alpha B^\pm \end{pmatrix} \cdot \begin{pmatrix}
\dbar^\alpha v^\pm \\
\dbar^\alpha B^\pm \end{pmatrix} \, {\rm d}x \, ,
\end{equation}
where $\lambda^\pm =\lambda(v^\pm,B^\pm)$ is a ${\mathcal C}^1$ function that will be chosen appropriately 
later on. In particular, the choice of $\lambda^\pm$ will be made so that we have
\begin{equation}
\label{lambdasmall}
\| \lambda^+ \|_{L^\infty([0,T'] \times \Omega^+)} <1 \, , \qquad 
\| \lambda^- \|_{L^\infty([0,T'] \times \Omega^-)} <1 \, ,
\end{equation}
which will imply that the matrix in the integrals defining $\H(t)$ is positive definite (hence we shall 
recover a control of the tangential derivatives of the solution).

We compute the time derivative
\begin{align}
\H'(t) =&\dfrac{1}{2} \, \sum_\pm \sum_{|\alpha| \le 3} \int_{\Omega^\pm} \begin{pmatrix}
0 & -\dt\lambda^\pm\\
-\dt\lambda^\pm & 0 \end{pmatrix} \begin{pmatrix}
\dbar^\alpha v^\pm \\
\dbar^\alpha B^\pm \end{pmatrix} \cdot \begin{pmatrix}
\dbar^\alpha v^\pm \\
\dbar^\alpha B^\pm \end{pmatrix} \, {\rm d}x \notag \\
&+\sum_\pm \sum_{|\alpha| \le 3} \int_{\Omega^\pm} \begin{pmatrix}
1 & -\lambda^\pm \\
-\lambda^\pm & 1 \end{pmatrix} \begin{pmatrix}
\dbar^\alpha\dt v^\pm \\
\dbar^\alpha\dt B^\pm \end{pmatrix} \cdot \begin{pmatrix}
\dbar^\alpha v^\pm \\
\dbar^\alpha B^\pm \end{pmatrix} \, {\rm d}x \notag \\
= &-\sum_\pm \sum_{|\alpha| \le 3} \int_{\Omega^\pm} \dt \lambda^\pm \, 
\dbar^\alpha v^\pm \cdot \dbar^\alpha B^\pm \, {\rm d}x \notag \\
&-\sum_\pm \sum_{|\alpha| \le 3} \int_{\Omega^\pm}\begin{pmatrix}
1 & -\lambda^\pm \\
-\lambda^\pm & 1 \end{pmatrix} \begin{pmatrix}
\dbar^\alpha \left\{ 
(\tilde v^\pm\cdot\nabla)v^\pm-(\tilde B^\pm\cdot\nabla)B^\pm  +A^T\nabla \, Q^\pm \right\}\\
\dbar^\alpha \left\{ (\tilde v^\pm\cdot\nabla)B^\pm-(\tilde B^\pm\cdot\nabla)v^\pm \right\}
\end{pmatrix} \cdot \begin{pmatrix}
\dbar^\alpha v^\pm \\
\dbar^\alpha B^\pm \end{pmatrix} \, {\rm d}x \notag \\
= &\sum_{p=1}^5 \H_p(t) \, ,\label{decompositionH'}
\end{align}
where each term $\H_p$ in the decomposition will be defined below, and we leave as a very simple exercise to 
the reader to check that the sum of all these terms coincides with the time derivative $\displaystyle\H'(t)$. 
We now define and estimate all the terms in the decomposition of $\H'(t)$. We first consider
\begin{equation*}
\H_1(t) := -\sum_\pm \sum_{|\alpha| \le 3} \int_{\Omega^\pm} \dt\lambda^\pm \,
\dbar^\alpha v^\pm \cdot \dbar^\alpha B^\pm \, {\rm d}x \, ,
\end{equation*}
which is trivially estimated by
\begin{equation}
\label{stimaH1}
\forall \, t \in [0,T'] \, ,\quad 
|\H_1(t)| \le C \, {\mathcal E}(t) \, \sum_\pm \|\dt \lambda^\pm\|_{L^\infty (\Omega^\pm)}\, .
\end{equation}
Next we consider some of the terms with the highest number of derivatives. Let us define
\begin{equation*}
\H_2(t) := -\sum_\pm \sum_{|\alpha| \le 3} \int_{\Omega^\pm} \begin{pmatrix}
1 & -\lambda^\pm \\
-\lambda^\pm & 1 \end{pmatrix} \begin{pmatrix}
(\tilde v^\pm\cdot\nabla) \dbar^\alpha v^\pm 
-(\tilde B^\pm\cdot\nabla) \dbar^\alpha B^\pm \\
(\tilde v^\pm\cdot\nabla) \dbar^\alpha B^\pm 
-(\tilde B^\pm\cdot\nabla) \dbar^\alpha v^\pm \end{pmatrix} \cdot \begin{pmatrix}
\dbar^\alpha v^\pm \\
\dbar^\alpha B^\pm \end{pmatrix} \, {\rm d}x \, .
\end{equation*}
This term is estimated by integrating by parts and recalling the boundary condition \eqref{bordo}. We obtain
\begin{multline*}
\H_2(t) = \sum_\pm \sum_{|\alpha| \le 3} \int_{\Omega^\pm} \dfrac{1}{2} \, 
\Big( \div \tilde v^\pm +\div (\lambda^\pm \tilde B^\pm) \Big) 
\left( |\dbar^\alpha v^\pm|^2 +|\dbar^\alpha B^\pm|^2 \right) \\ 
-\left( \div \tilde B^\pm +\div (\lambda^\pm \tilde v^\pm) \right) 
\dbar^\alpha v^\pm  \cdot \dbar^\alpha B^\pm \, {\rm d}x \, ,
\end{multline*}
from which we already get
\begin{equation*}
|\H_2(t)| \le C \, {\mathcal E}(t) \, \sum_\pm \|\div \tilde v^\pm,\div \tilde B^\pm\|_{L^\infty (\Omega^\pm)} 
+\|\div (\lambda^\pm \tilde v^\pm),\div (\lambda^\pm \tilde B^\pm)\|_{L^\infty (\Omega^\pm)} \, .
\end{equation*}
Using the expression of $\tilde v^\pm,\tilde B^\pm$, we get (recall that the estimate \eqref{uniforma} implies 
in particular $1+\dtre \psi \ge 1/2$)
\begin{align*}
\forall \, t \in [0,T'] \, ,\qquad \qquad \|\div \tilde v^\pm,\div \tilde B^\pm\|_{L^\infty (\Omega^\pm)} 
&\le F({\mathcal E}(t)) \, ,\\
\|\div (\lambda^\pm \tilde v^\pm),\div (\lambda^\pm \tilde B^\pm)\|_{L^\infty (\Omega^\pm)} &\le 
F({\mathcal E}(t)) \, \| \lambda^\pm \|_{W^{1,\infty}(\Omega^\pm)} \, .
\end{align*}
We thus end up with
\begin{equation}
\label{stimaH2}
\forall \, t \in [0,T'] \, ,\quad |\H_2(t)| \le F({\mathcal E}(t)) \, 
\Big( 1+ \sum_\pm \| \lambda^\pm \|_{W^{1,\infty}(\Omega^\pm)} \Big) \, .
\end{equation}

Let us now consider the term
\begin{align*}
\H_3(t):= &-\sum_\pm \sum_{|\alpha| \le 3} \int_{\Omega^\pm} \begin{pmatrix}
1 & -\lambda^\pm \\
-\lambda^\pm & 1 \end{pmatrix} \begin{pmatrix}
A^T\nabla \, (\dbar^\alpha Q^\pm) \\
0 \end{pmatrix} \cdot \begin{pmatrix}
\dbar^\alpha v^\pm \\
\dbar^\alpha B^\pm \end{pmatrix} \, {\rm d}x \\
= &-\sum_\pm \sum_{|\alpha| \le 3} \int_{\Omega^\pm} A^T\nabla \, (\dbar^\alpha Q^\pm) 
\cdot \left\{ \dbar^\alpha v^\pm -\lambda^\pm \, \dbar^\alpha B^\pm 
\right\} \, {\rm d}x \, .
\end{align*}
This is the term which requires the most careful analysis. We first observe that the term in the 
sum which corresponds to $\alpha=0$ (no tangential derivative) is estimated in an elementary way by 
Cauchy-Schwarz inequality, and admits an upper bound that is the same as in \eqref{stimaH2}. We thus 
feel free to slightly modify the definition of $\H_3$ and from now on we only consider the sum over 
the multi-indices $\alpha$ satisfying $1 \le |\alpha| \le 3$. A first integration by parts gives (here we 
use Einstein's convention over repeated indices)
\begin{align}
\H_3(t)= &\sum_{1 \le |\alpha| \le 3} \int_{\Gamma} A_{3i} \, \dbar^\alpha Q^+ \, 
\left\{ \dbar^\alpha v^+_i -\lambda^+ \, \dbar^\alpha B^+_i \right\} \, {\rm d}x' 
\notag \\
&-\sum_{1 \le |\alpha| \le 3} \int_{\Gamma_+} A_{3i} \, \dbar^\alpha Q^+ \, 
\left\{ \dbar^\alpha v^+_i -\lambda^+ \, \dbar^\alpha B^+_i \right\} \, {\rm d}x' 
\notag \\
&-\sum_{1 \le |\alpha| \le 3} \int_{\Gamma} A_{3i} \, \dbar^\alpha Q^- \, 
\left\{ \dbar^\alpha v^-_i -\lambda^- \, \dbar^\alpha B^-_i \right\} \, {\rm d}x' 
\label{defh3} \\
&+\sum_{1 \le |\alpha| \le 3} \int_{\Gamma_-} A_{3i} \, \dbar^\alpha Q^- \, 
\left\{ \dbar^\alpha v^-_i -\lambda^- \, \dbar^\alpha B^-_i \right\} \, {\rm d}x' 
\notag \\
&+\sum_\pm \sum_{1 \le |\alpha| \le 3} \int_{\Omega^\pm} \dbar^\alpha Q^\pm \, 
\partial_j \left\{ A_{ji} \, ( \dbar^\alpha v^\pm_i -\lambda^\pm \, 
\dbar^\alpha B^\pm_i) \right\} \, {\rm d}x \, .\notag
\end{align}
Let us notice first that 
\begin{equation*}
A_{3i} \, \{ \dbar^\alpha v^\pm_i -\lambda^\pm \, 
\dbar^\alpha B^\pm_i \}_{|x_3=\pm1} = \dfrac{1}{J} \, \{ 
\dbar^\alpha v^\pm_3 -\lambda^\pm \, \dbar^\alpha B^\pm_3 \}_{|x_3=\pm1} = 0 \, ,
\end{equation*}
because of \eqref{defAJa} and $\psi=v_3^\pm=B^\pm_3=0$ on $[0,T] \times \Gamma_{\pm}$. Therefore the second 
and fourth boundary integrals on $\Gamma_{\pm}$ in \eqref{defh3} vanish identically. As for the two boundary 
integrals on $\Gamma$, from \eqref{defAJa}, \eqref{bordo} and the boundary condition $[Q]=0$ on $\Gamma$ we 
have 
\begin{equation*}
A_{3 \cdot} =N \, ,\qquad [\dbar^\alpha Q]=0 \qquad{\rm on } \; \Gamma \, .
\end{equation*}
Therefore we may rewrite \eqref{defh3} as $\H_3(t) =\H_{31}(t)+\H_{32}(t)$ with
\begin{align}
\H_{31}(t) &:= \sum_{1\le|\alpha| \le 3} \int_\Gamma \dbar^\alpha Q \, \big[ 
(\dbar^\alpha v -\lambda \, \dbar^\alpha B)\cdot N \big] \, {\rm d}x' \, ,\label{h31} \\
\H_{32}(t) &:= \sum_\pm \sum_{1\le|\alpha| \le 3} \int_{\Omega^\pm} \dbar^\alpha Q^\pm \, 
\partial_j\left\{ A_{ji} \, (\dbar^\alpha v^\pm_i -\lambda^\pm \, 
\dbar^\alpha B^\pm_i) \right\} \, {\rm d}x \, ,\label{h32}
\end{align}
where $[\cdot]$ in \eqref{h31} still denotes the jump across $\Gamma$, and $Q$ denotes the common 
trace of $Q^\pm$ on $\Gamma$.

Let us first consider the term $ \H_{31}(t)$, which is where the choice of $\lambda^\pm$ is made. The 
boundary conditions $[v\cdot N] =B^\pm\cdot N =0$ on $\Gamma$ yield $\dbar^\alpha ([v\cdot N]) 
= \dbar^\alpha (B^\pm\cdot N) =0$ on $\Gamma$. Therefore we may write
\begin{equation*}
\H_{31}(t) =-\sum_{1 \le |\alpha| \le 3} \int_\Gamma \dbar^\alpha Q \, \Big[ 
[\dbar^\alpha;N] \cdot v -\lambda \,    [\dbar^\alpha;N] \cdot  B \Big] \, {\rm d}x' \, ,
\end{equation*}
where $[\dbar^\alpha;N]$ denotes the commutator between $\dbar^\alpha$ and the 
multiplication by $N$. This commutator can be written as a sum of the form
\begin{equation*}
[\dbar^\alpha;N] =\dbar^\alpha N +\sum_{1 \le |\beta| \le |\alpha|-1} \star 
\, \dbar^\beta N \, \dbar^{\alpha-\beta} \, ,
\end{equation*}
where $\star$ denotes some harmless numerical coefficient. Let us assume for the time being that we can 
construct $\lambda^\pm$ on $[0,T'] \times \Gamma$ that satisfy
\begin{equation}
\label{linearsystem}
\begin{cases}
\lambda^+ \, B_1^+-\lambda^- \, B^-_1=[v_1] \, ,\\
\lambda^+ \, B_2^+-\lambda^- \, B^-_2=[v_2] \ ,
\end{cases}
\end{equation}
so that $[v'-\lambda B']=0$, where we have set $v':=(v_1,v_2)$ and so on. Then the decomposition of the commutator reduces $\H_{31}(t)$ to
\begin{equation*}
\H_{31}(t) =\sum_{1 \le |\alpha| \le 3} \sum_{1 \le |\beta| \le |\alpha|-1} \star 
\int_\Gamma \dbar^\alpha Q \, \dbar^\beta \nabla' f \cdot \Big( 
\dbar^{\alpha-\beta} v' -\lambda \, \dbar^{\alpha-\beta} B' \Big) \, {\rm d}x' \, ,
\end{equation*}
where we have set $\nabla':=(\duno,\ddue)$ (here the indices $\pm$ do not 
play any role so we feel free to omit them). We now recall the following classical product estimate.

\begin{lemma}
\label{lemma4}
The product mapping $H^{0.5}(\T^2) \times H^{1.5} (\T^2) \longrightarrow H^{0.5}(\T^2)$, $(f,g) \longmapsto 
f\, g$ is continuous.
\end{lemma}

We can now estimate each term in the above decomposition of $\H_{31}(t)$. In the case $|\alpha|-|\beta| 
= 1$, we get (use Lemma \ref{lemma4} for the product estimate and the fact that $H^{1.5}(\T^2)$ is an algebra)
\begin{align*}
\left| \int_{\Gamma} \dbar^\alpha Q \, \dbar^\beta \nabla' f \cdot \Big( 
\dbar^{\alpha-\beta} v' -\lambda \, \dbar^{\alpha-\beta} B' \Big) \, {\rm d}x' \right| 
\le C \, \left\| \dbar^\alpha Q \right\|_{H^{-0.5}(\Gamma)} \left\| 
\dbar^\beta \nabla' f \cdot \big( \dbar^{\alpha-\beta} v' 
-\lambda \, \dbar^{\alpha-\beta} B' \big) \right\|_{H^{0.5}(\Gamma)} &
\\
\le C \, \left\| \nabla Q \right\|_{H^{1.5}(\Gamma)} \, 
\left\| \dbar^\beta \nabla' f \right\|_{H^{0.5}(\T^2)} \, 
\left\| \dbar^{\alpha-\beta} v' -\lambda \, \dbar^{\alpha-\beta} B' 
\right\|_{H^{1.5}(\Gamma)} &
\\
\le F({\mathcal E}(t)) \, \Big( 1+\sum_\pm \| \lambda^\pm \|_{H^{1.5}(\Gamma)} \Big) \, .&
\end{align*}
In the case $|\alpha|-|\beta| \ge 2$, which only happens for $|\alpha|=3$ and $|\beta|=1$, we have
\begin{align*}
\left| \int_{\Gamma} \dbar^\alpha Q \, \dbar^\beta \nabla' f \cdot \Big( 
\dbar^{\alpha-\beta} v' -\lambda \, \dbar^{\alpha-\beta} B' \Big) \, {\rm d}x' \right| 
\le C \, \left\| \dbar^\alpha Q \right\|_{H^{-0.5}(\Gamma)} \left\| 
\dbar^\beta \nabla' f \cdot \big( \dbar^{\alpha-\beta} v' 
-\lambda \, \dbar^{\alpha-\beta} B' \big) \right\|_{H^{0.5}(\Gamma)} &\\
\le C \, \left\| \nabla Q \right\|_{H^{1.5}(\Gamma)} \, 
\left\| \dbar^\beta \nabla' f \right\|_{H^{1.5}(\T^2)} \, 
\left\| \dbar^{\alpha-\beta} v' -\lambda \, \dbar^{\alpha-\beta} B' 
\right\|_{H^{0.5}(\Gamma)} &\\
\le F({\mathcal E}(t)) \, \Big( 1+\sum_\pm \| \lambda^\pm \|_{H^{1.5}(\Gamma)} \Big) \, .&
\end{align*}
Summing all the estimates, we have obtained
\begin{equation}
\label{stimaH31}
\forall \, t \in [0,T'] \, ,\quad |\H_{31}(t)| \le 
F({\mathcal E}(t)) \, \Big( 1+\sum_\pm \| \lambda^\pm \|_{H^{1.5}(\Gamma)} \Big) \, ,
\end{equation}
provided that we can construct $\lambda^\pm$ that satisfy \eqref{linearsystem}. Let us therefore turn to 
the construction of these functions.

We first observe that the boundary conditions \eqref{bordo} give
\begin{equation*}
B_3^\pm = B_1^\pm \, \duno f +B_2^\pm \, \ddue f \, ,\quad 
[v_3] = [v_1] \, \duno f +[v_2] \, \ddue f \, ,\quad \text{\rm on } \; \Gamma \, ,
\end{equation*}
so \eqref{linearsystem} is equivalent to the relation
\begin{equation*}
[v] =\lambda^+ \, B^+ -\lambda^- \, B^- \quad \text{\rm on } \; \Gamma \, .
\end{equation*}
Using the lower bound \eqref{uniformc} on the time interval $[0,T']$, we know that \eqref{linearsystem} 
is a Cramer system (otherwise, $B^+$ and $B^-$ would be colinear). Hence $\lambda^\pm$ are uniquely 
determined on $[0,T'] \times \Gamma$ and have the same regularity as $v^\pm,B^\pm$ on the boundary $\Gamma$. 
Moreover, the latter relations give
\begin{equation*}
|\lambda^\pm (t,x',0)| =\dfrac{|B^\mp \times [v]|}{|B^+ \times B^-|} (t,x',0) \le 1-\dfrac{\delta_0}{2} \, ,
\end{equation*}
where we have used \eqref{uniformd}. As in \cite{trakhinin05,trakhinin052}, we extend $\lambda^\pm$ to 
the domains $\Omega^\pm$ as functions that do not depend on the normal variable $x_3$. Using time or 
tangential differentiation on the system \eqref{linearsystem}, we can easily obtain the estimates
\begin{align}
\forall \, t \in [0,T'] \, ,\quad \| \lambda^\pm \|_{H^{1.5}(\Gamma)} 
+\| \lambda^\pm \|_{W^{1,\infty}(\Omega^\pm)} +\| \dt \lambda^\pm \|_{L^\infty(\Omega^\pm)} &\le 
F({\mathcal E}(t)) \, ,\label{stimalambdapm} \\
\| \lambda^\pm \|_{L^\infty(\Omega^\pm)} &\le 1-\dfrac{\delta_0}{2} \, .\notag
\end{align}
The latter estimates on $\lambda^\pm$ simplify \eqref{stimaH1}, \eqref{stimaH2} and \eqref{stimaH31}, 
and give
\begin{equation}
\label{stimaH1231'}
\forall \, t \in [0,T'] \, ,\quad |\H_1(t)| +|\H_2(t)| +|\H_{31}(t)| \le F({\mathcal E}(t)) \, .
\end{equation}
We emphasize that in the estimate \eqref{stimaH1231'}, the nondecreasing function $F$ depends on $\delta_0$ 
because the estimates on $\lambda^\pm$ depend on $\delta_0$, but $F$ does not depend on the particular 
solution that we are considering.

Let us now consider the term $\H_{32}(t)$ in \eqref{h32}. We decompose $\H_{32}(t)$ as $\H_{32}(t) = 
\H_{321}(t)+\H_{322}(t)$, with
\begin{align*}
\H_{321}(t) &:= \sum_\pm \sum_{1\le|\alpha| \le 3} \int_{\Omega^\pm} \dbar^\alpha Q^\pm \, 
(\partial_j A_{ji}) \, (\dbar^\alpha v^\pm_i -\lambda^\pm \, \dbar^\alpha B^\pm_i) 
\, {\rm d}x \, ,\\
\H_{322}(t) &:= \sum_\pm \sum_{1\le|\alpha| \le 3} \int_{\Omega^\pm} \dbar^\alpha Q^\pm \, 
A_{ji} \, \partial_j (\dbar^\alpha v^\pm_i -\lambda^\pm \, \dbar^\alpha B^\pm_i) 
\, {\rm d}x \, .
\end{align*}
The first term $\H_{321}(t)$ is easily estimated by applying Cauchy-Schwarz inequality and by using 
the $L^\infty$ estimate of $\lambda^\pm$, see \eqref{stimalambdapm}:
\begin{equation}
\label{stimaH321}
\forall \, t \in [0,T'] \, ,\quad | \H_{321}(t)| \le F({\mathcal E}(t)) \, .
\end{equation}
As for $\H_{322}(t)$, since we have the divergence constraint $A_{ji} \, \partial_j v^\pm_i 
=A_{ji} \, \partial_j B^\pm_i =0$, we may write
\begin{equation*}
\H_{322}(t) = -\sum_\pm \sum_{1\le|\alpha|\le3} \int_{\Omega^\pm} \dbar^\alpha Q^\pm \left\{ 
[\dbar^\alpha;A_{ji} \, \partial_j] v^\pm_i 
+A_{ji} \, (\partial_j \lambda^\pm) \, \dbar^\alpha B^\pm_i 
-\lambda^\pm \, [\dbar^\alpha;A_{ji} \, \partial_j] B^\pm_i \right\} \, {\rm d}x \, ,
\end{equation*}
where $[ \cdot ; \cdot ]$ still denotes the commutator. The latter terms are now estimated in a somehow 
brutal way by applying Cauchy-Schwarz inequality. We recall that the $H^4$ norm of $\psi$ is controlled 
by the $H^{3.5}$ norm of $f$ thanks to Lemma \ref{lemma1}, and that commutators in $L^2$ are controlled by 
standard estimates which may be found for instance in \cite[page 295]{benzoni-serre}. Eventually we obtain
\begin{equation*}
\forall \, t \in [0,T'] \, ,\quad |\H_{322}(t)| \le F({\mathcal E}(t)) \, .
\end{equation*}
Combining with \eqref{stimaH321}, and \eqref{stimaH1231'}, we end up with
\begin{equation}
\label{stimaH3}
\forall \, t \in [0,T'] \, ,\quad |\H_1(t)| +|\H_2(t)| +|\H_{3}(t)| \le F({\mathcal E}(t)) \, .
\end{equation}

Going on with the estimate of the terms in the decomposition \eqref{decompositionH'} of $\H'(t)$, we finally 
consider
\begin{equation*}
\H_4(t) := -\sum_\pm \sum_{|\alpha| \le 3} \int_{\Omega^\pm} \begin{pmatrix}
1 & -\lambda^\pm \\
-\lambda^\pm & 1 \end{pmatrix} \begin{pmatrix}
[\dbar^\alpha;\tilde v^\pm\cdot\nabla] v^\pm 
-[\dbar^\alpha;\tilde B^\pm\cdot\nabla] B^\pm \\
[\dbar^\alpha;\tilde v^\pm\cdot\nabla] B^\pm 
-[\dbar^\alpha;\tilde B^\pm\cdot\nabla] v^\pm \end{pmatrix} \cdot \begin{pmatrix}
\dbar^\alpha v^\pm \\
\dbar^\alpha B^\pm \end{pmatrix} \, {\rm d}x \, ,
\end{equation*}
and
\begin{equation*}
\H_5(t) := -\sum_\pm \sum_{|\alpha| \le 3} \int_{\Omega^\pm} 
\Big[ \dbar^\alpha;A^T \nabla \Big] \, Q^\pm \cdot \left\{ 
\dbar^\alpha v^\pm -\lambda^\pm \, \dbar^\alpha B^\pm \right\} \, {\rm d}x \, .
\end{equation*}
Indeed the reader can check that the relation \eqref{decompositionH'} holds with the above definitions 
of $\H_1,\dots,\H_5$. Applying again the classical commutator estimates and using once again the $L^\infty$ 
estimates of $\lambda^\pm$, we have
\begin{equation}
\label{stimaH4}
\forall \, t \in [0,T'] \, ,\quad |\H_4(t)| +|\H_5(t)| \le F({\mathcal E}(t)) \, .
\end{equation}

Combining \eqref{stimaH3} and \eqref{stimaH4}, we have therefore derived the inequality
\begin{equation*}
\forall \, t \in [0,T'] \, ,\quad |\H'(t)| \le F({\mathcal E}(t)) \, ,
\end{equation*}
for a given nonnegative nondecreasing function $F$ that is independent of the solution. Integrating from 
$0$ to $t \in [0,T']$ and using the $L^\infty$ bounds on $\lambda^\pm$, we have already proved our main 
a priori estimate for tangential derivatives:
\begin{equation}
\label{stima-tan}
\forall \, t \in [0,T'] \, ,\quad \sum_{|\alpha| \le 3} 
\big\| \dbar^\alpha v^\pm(t),\dbar^\alpha B^\pm(t) \big\|_\pm^2 
\le M_0 +t \, F(\max_{0 \le s \le t}{\mathcal E}(s)) \, ,
\end{equation}
where $M_0$ is a numerical constant that only depends on $\delta_0$ and $R$ (here we have used 
\eqref{stimalambdapm} to derive a lower bound for the positive definite matrix appearing in the 
definition of the energy functional $\H$).

\section{Divergence and curl estimates for $v$ and $B$}

\subsection{Estimates for the divergence}

In this section we derive suitable estimates for the divergence of $v^\pm,B^\pm$ in $\Omega^\pm$. 
Expanding the divergence constraint for $v^\pm$, we find that for each $t\in[0,T']$, there holds
\begin{equation*}
\duno  v_1^\pm -\dfrac{\duno\psi}{J} \, \dtre v_1^\pm +\ddue v_2^\pm -\dfrac{\ddue\psi}{J} \, \dtre v_2^\pm 
+\dfrac{1}{J} \, \dtre v_3^\pm =0 \quad {\rm in } \; \Omega^\pm \, ,
\end{equation*}
from which the identity
\begin{equation*}
{\rm div} \,v^\pm =\dfrac{\nabla \psi \cdot \dtre v^\pm}{J} \quad {\rm in } \; \Omega^\pm \, 
\end{equation*}
readily follows. Since $H^2(\Omega^\pm)$ is an algebra, we get
\begin{equation*}
\forall \, t \in [0,T'] \, ,\quad \| {\rm div} \, v^\pm(t) \|_{2,\pm} \le 
C \, \left\| \dfrac{\nabla \psi}{J} (t) \right\|_2 \, \| \dtre v^\pm(t) \|_{2,\pm} 
\le C \, \| f(t) \|_{H^{2.5}(\T^2)} \, \| v^\pm(t) \|_{3,\pm} \, .
\end{equation*}
The analogue estimate for the divergence of $B^\pm$ is obtained by following the same lines, and we have 
thus proved the a priori estimate
\begin{equation}
\label{div-est}
\forall \, t \in [0,T'] \, ,\quad \| {\rm div} \, v^\pm(t),{\rm div} \, B^\pm(t) \|_{2,\pm} \le 
C_0 \, \| f(t) \|_{H^{2.5}(\T^2)} \, \| v^\pm(t),B^\pm(t) \|_{3,\pm} \, .
\end{equation}

\subsection{Estimates for the curl}

In order to estimate the curl of $v^\pm,B^\pm$ we proceed as follows. Let us introduce the curl of 
the Eulerian velocity and magnetic fields $u,H$
\begin{equation*}
\tilde\zeta := {\rm curl} \, u \, ,\quad \tilde\xi := {\rm curl} \, H \, ,
\end{equation*}
and set
\begin{equation}
\label{lagrangeancurl}
\begin{cases}
\zeta := \tilde\zeta \circ \Psi =({\rm curl} \, u) \circ \Psi 
= (A^T \nabla) \times (u \circ \Psi) = (A^T \nabla) \times v \, ,\\
\xi := \tilde\xi \circ \Psi = ({\rm curl} \, H) \circ \Psi 
= (A^T \nabla) \times (H \circ \Psi) = (A^T \nabla) \times B \, .
\end{cases}
\end{equation}
Using the definition of the matrix $A$ in \eqref{defAJa}, the relations \eqref{lagrangeancurl} can be 
easily inverted to find
\begin{equation}\label{lagrangeancurl-inv}
{\rm curl} \, v =\zeta +\dfrac{\nabla \psi \times \dtre v}{J} \, ,\quad 
{\rm curl} \, B =\xi +\dfrac{\nabla \psi \times \dtre B}{J} \, .
\end{equation}
Applying the curl operator to the original equations \eqref{euler} satisfied by $(u,H)$, we easily find 
that the Eulerian curls $(\tilde\zeta,\tilde\xi)$ solve the system
\begin{equation*}
\begin{cases}
\dt \tilde\zeta^\pm +(u^\pm \cdot \nabla) \tilde\zeta^\pm -(H^\pm \cdot \nabla) \tilde\xi^\pm 
-(\tilde\zeta^\pm \cdot \nabla) u^\pm +(\tilde\xi^\pm \cdot \nabla) H^\pm =0 \, ,\\
\dt \tilde\xi^\pm +(u^\pm \cdot \nabla) \tilde\xi^\pm -(H^\pm \cdot \nabla) \tilde\zeta^\pm 
+[{\rm curl};u^\pm \cdot \nabla] H^\pm -[{\rm curl};H^\pm \cdot \nabla] u^\pm =0 \, ,
\end{cases}
\end{equation*}
in $\bigcup\limits_{t\in[0,T]} \{ t \} \times \Omega^\pm(t)$. Making use of \eqref{lagrangeancurl} and 
recalling the definitions in \eqref{defNtilde}, it follows that $(\zeta,\xi)$ solve
\begin{equation}
\label{equcurl}
\begin{cases}
\dt \zeta^\pm + (\tilde{v}^\pm \cdot \nabla) \zeta^\pm -(\tilde{B}^\pm \cdot \nabla) \xi^\pm 
-(A \, \zeta^\pm \cdot \nabla) v^\pm +(A\, \xi^\pm \cdot \nabla)B^\pm =0 \, ,\\
\dt \xi^\pm +(\tilde{v}^\pm \cdot \nabla) \xi^\pm -(\tilde{B}^\pm \cdot \nabla) \zeta^\pm 
+[A^T \nabla \times;A\, v^\pm \cdot \nabla] B^\pm -[A^T \nabla \times;A\, B^\pm \cdot \nabla] v^\pm =0 \, ,
\end{cases}
\end{equation}
in $[0,T] \times \Omega^\pm$. Thus, in order to estimate the curl of $v^\pm,B^\pm$, we are reduced, after 
\eqref{lagrangeancurl-inv}, to proving suitable bounds for the $H^2-$norm of the solution $(\zeta,\xi)$ 
to \eqref{equcurl}. Let us observe that with our regularity assumptions on the original solution, there 
holds $(\zeta,\xi) \in {\mathcal C}^1(H^2) \cap {\mathcal C}(H^3)$ so all integration by parts below are 
legitimate.

Let us introduce an associated energy functional ${\mathcal K}$ defined by
\begin{equation}
\label{energy-curl}
{\mathcal K}(t) := \dfrac{1}{2} \, \sum_\pm \sum_{|\beta| \le 2} \int_{\Omega^\pm} \left\{ 
|\partial^\beta \zeta^\pm(t)|^2 +|\partial^\beta \xi^\pm(t)|^2\right\} \, {\rm d}x \, .
\end{equation}
Differentiating with respect to $t$ and making use of \eqref{defAJa}, \eqref{defNtilde}, \eqref{equcurl} gives
\begin{equation}
\label{energy-curl1}
{\mathcal K}'(t) =\sum_\pm \sum_{|\beta| \le 2} \int_{\Omega^\pm} \left\{ 
\partial^\beta \dt \zeta^\pm \cdot \partial^\beta \zeta^\pm 
+\partial^\beta \dt \xi^\pm \cdot \partial^\beta \xi^\pm \right\} \, {\rm d}x 
={\mathcal K}_1(t) +{\mathcal K}_2(t) +{\mathcal K}_3(t) \, ,
\end{equation}
where
\begin{align*}
{\mathcal K}_1(t) := -\sum_\pm \sum_{|\beta| \le 2} \int_{\Omega^\pm} &\left\{ 
(\tilde{v}^\pm \cdot \nabla) \partial^\beta \zeta^\pm  
-(\tilde{B}^\pm \cdot \nabla) \partial^\beta \xi^\pm \right\} \cdot \partial^\beta \zeta^\pm \\
&+\left\{ (\tilde{v}^\pm \cdot \nabla) \partial^\beta \xi^\pm 
-(\tilde{B}^\pm \cdot \nabla) \partial^\beta \zeta^\pm \right\} \cdot \partial^\beta\xi^\pm \, {\rm d}x\, ,\\
{\mathcal K}_2(t) := -\sum_\pm \sum_{|\beta| \le 2} \int_{\Omega^\pm} &\left\{ 
[\partial^\beta;\tilde{v}^\pm \cdot \nabla] \zeta^\pm 
-[\partial^\beta;\tilde{B}^\pm \cdot \nabla] \xi^\pm \right\} \cdot \partial^\beta \zeta^\pm \\
&+\left\{ [\partial^\beta;\tilde{v}^\pm \cdot \nabla] \xi^\pm 
-[\partial^\beta;\tilde{B}^\pm \cdot \nabla] \zeta^\pm \right\} \cdot \partial^\beta \xi^\pm \, {\rm d}x \, ,\\
{\mathcal K}_3(t) := -\sum_\pm \sum_{|\beta| \le 2} \int_{\Omega^\pm} 
&\partial^\beta \left( (A\xi^\pm \cdot \nabla) B^\pm -(A\zeta^\pm \cdot \nabla) v^\pm \right) \cdot 
\partial^\beta \zeta^\pm \\
&+\partial^\beta \left( \left[ A^T \nabla \times;Av^\pm \cdot \nabla \right] B^\pm 
-\left[ A^T \nabla \times;AB^\pm \cdot \nabla \right] v^\pm \right) \cdot \partial^\beta \xi^\pm \, {\rm d}x \, .
\end{align*}
Let us estimate separately each of the above ${\mathcal K}_i$, for $i=1,2,3$. We start with ${\mathcal K}_1$. 
To estimate this term, we use Leibniz' rule and integrate by parts. The boundary conditions \eqref{bordo} give
\begin{align*}
{\mathcal K}_1(t) &= -\dfrac{1}{2} \, \sum_\pm \sum_{|\beta| \le 2} \int_{\Omega^\pm} 
\left\{ \tilde{v}^\pm \cdot \nabla \left( |\partial^\beta\zeta^\pm|^2 +|\partial^\beta\xi^\pm|^2 \right) 
-2\, \tilde{B}^\pm \cdot \nabla \left( \partial^\beta \xi^\pm \cdot \partial^\beta \zeta^\pm \right) 
\right\} \, {\rm d}x \\
&= \sum_\pm \sum_{|\beta| \le 2} \int_{\Omega^\pm} \left\{ \dfrac{1}{2} \, {\rm div} \, \tilde{v}^\pm \, 
\left( |\partial^\beta \zeta^\pm|^2 +|\partial^\beta \xi^\pm|^2 \right) -{\rm div} \, \tilde{B}^\pm \, 
\partial^\beta \xi^\pm \cdot \partial^\beta \zeta^\pm \right\} \, {\rm d}x \, .
\end{align*}
Applying Cauchy-Schwarz inequality, we obtain
\begin{equation}
\label{curl-est1}
\forall \, t \in [0,T'] \, ,\quad |{\mathcal K}_1(t)| \le F({\mathcal E}(t)) \, .
\end{equation}

Let us now deal with the term ${\mathcal K}_2$. We focus on the first integral involved in the definition 
of ${\mathcal{K}}_2$, namely
\begin{equation*}
\sum_{|\beta| \le 2} \int_{\Omega^\pm} [\partial^\beta;\tilde v^\pm \cdot \nabla] \zeta^\pm 
\cdot \partial^\beta \zeta^\pm \, {\rm d}x \, .
\end{equation*}
In the sequel $\partial^1$ and $\partial^2$ stand for any derivative of order one and order two respectively. The commutator is zero if $\beta=0$. If $|\beta|=1$, the integral is of the form
\begin{equation*}
\int_{\Omega^\pm} \partial^1 \zeta^\pm \, \partial^1 \tilde v^\pm \, \partial^1 \zeta^\pm \, {\rm d}x \, .
\end{equation*}
Using 
an $L^\infty$ bound for $\partial^1 \tilde v^\pm$ and Cauchy-Schwarz for the two remaining terms, we have
\begin{equation*}
\left| \int_{\Omega^\pm} \partial^1 \zeta^\pm \, \partial^1 \tilde v^\pm \, \partial^1 \zeta^\pm \, {\rm d}x 
\right| \le F({\mathcal E}(t)) \, .
\end{equation*}
It remains to examine the terms in the commutator with $|\beta| =2$. We can easily check that such a commutator 
can be rewritten as a sum of the form (we omit the harmless numerical constants)
\begin{equation*}
\int_{\Omega^\pm} \partial^1 \tilde v^\pm \, \partial^2 \zeta^\pm \, \partial^2 \zeta^\pm 
+\partial^2 \tilde v^\pm \, \partial^1 \zeta^\pm \, \partial^2 \zeta^\pm \, {\rm d}x \, .
\end{equation*}
The first term is estimated as in the case $|\beta|=1$ by using an $L^\infty$ bound for $\partial^1 \tilde v^\pm$. 
The second of these two terms requires more attention. We combine H\"older's inequality and the Sobolev imbedding 
Theorem (recall that in three space dimensions $H^1$ is imbedded in $L^6$):
\begin{equation*}
\left| \int_{\Omega^\pm} \partial^2 \tilde v^\pm \, \partial^1 \zeta^\pm \, \partial^2 \zeta^\pm \, {\rm d}x 
\right| \le |\partial^2 \tilde v^\pm|_{3,\pm} \, |\partial^1 \zeta^\pm|_{6,\pm} \, 
\| \partial^2 \zeta^\pm \|_{\pm} 
\le C \, \| \tilde v^\pm \|_{3,\pm} \, \| \zeta^\pm \|_{2,\pm}^2 \le F({\mathcal E}(t)) \, .
\end{equation*}
In a completely similar way, we can handle the other commutators in ${\mathcal K}_2(t)$ to finally get the 
estimate
\begin{equation}
\label{curl-est2}
\forall \, t \in [0,T'] \, ,\quad |{\mathcal K}_2(t)| \le F({\mathcal E}(t)) \, .
\end{equation}

We now turn to the last term ${\mathcal K}_3$, that we write in the form ${\mathcal K}_3(t) 
={\mathcal K}_{31}(t) +{\mathcal K}_{32}(t)$ with
\begin{align*}
{\mathcal K}_{31}(t) &:= -\sum_\pm \sum_{|\beta| \le 2} \int_{\Omega^\pm} 
\partial^\beta \left( (A\xi^\pm \cdot \nabla) B^\pm -(A\zeta^\pm \cdot \nabla) v^\pm \right) 
\cdot \partial^\beta \zeta^\pm \, {\rm d}x \, ,\\
{\mathcal K}_{32}(t) &:= -\sum_\pm \sum_{|\beta| \le 2} \int_{\Omega^\pm} 
\partial^\beta \left\{ [A^T \nabla \times;Av^\pm \cdot \nabla] B^\pm 
-[A^T \nabla \times;AB^\pm \cdot \nabla] v^\pm \right\} \cdot \partial^\beta \xi^\pm \, {\rm d}x \, .
\end{align*}
The first integral in ${\mathcal K}_{31}(t)$ are estimated by Cauchy-Schwarz inequality and by using the fact 
that $H^2(\Omega^\pm)$ is an algebra:
\begin{multline*}
\left| \int_{\Omega^\pm} \partial^\beta \left( (A\xi^\pm \cdot \nabla) B^\pm \right) \cdot 
\partial^\beta \zeta^\pm \, {\rm d}x \right| 
\le \| \zeta^\pm \|_{2,\pm} \, \| (A\xi^\pm \cdot \nabla) B^\pm \|_{2,\pm} \\
\le \| \zeta^\pm \|_{2,\pm} \, \| A \|_2 \, \| \xi^\pm \|_{2,\pm} \, \| B^\pm \|_{3,\pm} 
\le F({\mathcal E}(t)) \, .
\end{multline*}
The second integral in ${\mathcal K}_{31}(t)$ is estimated in the same way and we get
\begin{equation}
\label{curl-est3}
\forall \, t \in [0,T'] \, ,\quad |{\mathcal K}_{31}(t)| \le F({\mathcal E}(t)) \, .
\end{equation}
As for ${\mathcal K}_{32}(t)$, it is rather easy to see that the quantity  $[A^T \nabla \times;Av^\pm 
\cdot \nabla] B^\pm -[A^T \nabla \times;AB^\pm \cdot \nabla] v^\pm$ can be expanded as a sum of terms 
of the form
\begin{equation*}
A \, \partial^1 A \, v^\pm \, \partial^1 B^\pm +A\, \partial^1 A \, B^\pm \, \partial^1 v^\pm 
+A\, A\, \partial^1 v^\pm \, \partial^1 B^\pm \, ,
\end{equation*}
where we have disregarded the indices for the sake of simplicity. Hence the $H^2$ norm of this quantity 
can be estimated by a quantity of the form $F({\mathcal E}(t))$. Using Cauchy-Schwarz inequality in 
${\mathcal K}_{32}(t)$, we end up with
\begin{equation*}
\forall \, t \in [0,T'] \, ,\quad |{\mathcal K}_{32}(t)| \le F({\mathcal E}(t)) \, .
\end{equation*}
Combining the latter estimate with \eqref{curl-est1}, \eqref{curl-est2} and \eqref{curl-est3}, we have 
obtained
\begin{equation*}
\forall \, t \in [0,T'] \, ,\quad |{\mathcal K}'(t)| \le F({\mathcal E}(t)) \, .
\end{equation*}
We can now integrate this inequality from $0$ to $t$ and use \eqref{lagrangeancurl-inv}. The ``error'' terms 
$\nabla \psi \times \dtre v_3^\pm/J$, $\nabla \psi \times \dtre B_3^\pm/J$ are estimated as in the paragraph on the divergence estimate, see 
\eqref{div-est}, so eventually we get
\begin{multline}
\label{curl-est}
\forall \, t \in [0,T'] \, ,\quad \| {\rm curl} \, v^\pm(t),{\rm curl} \, B^\pm(t) \|_{2,\pm}^2 \le M_0 \\
+t \, F(\max_{0 \le s \le t}{\mathcal E}(s)) 
+C_0 \, \| f(t) \|_{H^{2.5}(\T^2)}^2 \, \| v^\pm(t),B^\pm(t) \|_{3,\pm}^2 \, .
\end{multline}

\subsection{Final estimate for the velocity and magnetic field}

With the above divergence and curl estimates, we are ready to obtain the main a priori estimate for the 
velocity and magnetic field in each domain $\Omega^\pm$. The only point is to observe, through elementary 
algebraic manipulations, that the $H^3$ norm of a vector field is controlled by the $L^2$ norms of tangential 
derivatives of order $\le 3$ and by the $H^2$ norms of its divergence and of its curl. We thus add the 
estimates \eqref{stima-tan}, \eqref{div-est} and \eqref{curl-est} to obtain
\begin{equation*}
\forall \, t \in [0,T'] \, ,\quad \| v^\pm(t),B^\pm(t) \|_{3,\pm}^2 \le M_0 
+t \, F(\max_{0 \le s \le t}{\mathcal E}(s)) 
+C_0 \, \| f(t) \|_{H^{2.5}(\T^2)}^2 \, \| v^\pm(t),B^\pm(t) \|_{3,\pm}^2 \, ,
\end{equation*}
where, of course, the numerical constants $M_0,C_0$ are independent of the solution. Consequently, up to 
choosing $\eps_0$ small enough so that $C_0 \, \eps_0 \le 1/2$ and adapting the time interval $[0,T']$ so 
that \eqref{uniforma} is valid with the new definition of $\eps_0$, we obtain
\begin{equation}
\label{stima-vB}
\forall \, t \in [0,T'] \, ,\quad \| v^\pm(t),B^\pm(t) \|_{3,\pm}^2 \le M_0 
+t \, F(\max_{0 \le s \le t}{\mathcal E}(s)) \, .
\end{equation}

\section{Estimate of the front}

From the linear system of the boundary conditions on $\Gamma$
\begin{equation}
\label{systemgradf}
\begin{cases}
B_1^+ \, \duno f +B^+_2 \, \ddue f =B_3^+ \, ,\\
B_1^- \, \duno f +B^-_2 \, \ddue f =B_3^- \, ,
\end{cases}
\end{equation}
we have already seen that the determinant $B_1^+ \, B_2^- -B_2^+ \, B_1^-$ does not vanish on 
$[0,T'] \times \Gamma$. More precisely, we have
\begin{equation*}
|B_1^+ \, B_2^- -B_2^+ \, B_1^- \, (t,x',0)|^2 = 
\dfrac{|B^+ \times B^- (t,x',0)|^2}{1+|\nabla' f (t,x')|^2} \ge \dfrac{\delta_0^2}{4\, (1+C \, \eps_0^2)} \, ,
\end{equation*}
where we have used \eqref{uniformc}, \eqref{uniforma} and the imbedding $H^{1.5}(\T^2) \hookrightarrow 
L^\infty(\T^2)$. We also note that thanks to \eqref{uniformb}, the $L^\infty$ norm of $B^\pm$ is uniformly 
controlled on $[0,T']$. Therefore, using the latter uniform bound for the determinant and inverting the 
linear system \eqref{systemgradf}, we have
\begin{equation}
\label{stimaf1}
\forall \, t \in [0,T'] \, ,\quad \| \nabla' f (t) \|_{H^{2.5}(\T^2)} \le C_0 \, \| B^\pm (t) \|_{3,\pm} \, ,
\end{equation}
with $C_0$ depending only on $\delta_0$ and $R$.

From the other boundary conditions on $\Gamma$:
\begin{equation*}
\dt f =v_3^\pm -v_1^\pm \, \duno f -v^\pm_2 \, \ddue f \, ,
\end{equation*}
\eqref{stimaf1} and the fact that $H^{2.5}(\T^2)$ is an algebra, we infer the second main estimate for $f$:
\begin{equation}
\label{stimaf2}
\forall \, t \in [0,T'] \, ,\quad 
\| \dt f (t) \|_{H^{2.5}(\T^2)} \le C_0 \, \left( \| v^\pm(t) \|_{3,\pm} +\| v^\pm(t),B^\pm (t) \|_{3,\pm}^2 
\right) \, .
\end{equation}
In particular, we can integrate from $0$ to $t$ and get
\begin{equation}
\label{stimaf3}
\forall \, t \in [0,T'] \, ,\quad 
\| f(t) \|_{H^{2.5}(\T^2)} \le \| f_0 \|_{H^{2.5}(\T^2)} +t \, F(\max_{0 \le s \le t} {\mathcal E}(s)) \, .
\end{equation}
We simplify \eqref{stimaf1}, \eqref{stimaf2} and \eqref{stimaf3} by using \eqref{stima-vB} (we feel free 
to use $t^2 \le t$ which always holds by assuming, without loss of generality $T' \le 1$):
\begin{align}
\forall \, t \in [0,T'] \, ,\quad 
\| \dt f(t) \|_{H^{2.5}(\T^2)}^2 &\le M_0 +t \, F(\max_{0 \le s \le t} {\mathcal E}(s)) \, ,\notag \\
\|f(t)\|_{H^{3.5}(\T^2)}^2 &\le M_0 +t \, F(\max_{0 \le s \le t} {\mathcal E}(s)) \, ,\label{stimaf4} \\
\| f(t) \|_{H^{2.5}(\T^2)} &\le \| f_0 \|_{H^{2.5}(\T^2)} +t \, F(\max_{0 \le s \le t} {\mathcal E}(s)) 
\, .\notag
\end{align}
The last estimate in \eqref{stimaf4} says that $f(t)$ remains small in $H^{2.5}$ provided that we start 
from small initial data and the first and second estimates in \eqref{stimaf4} give a control of $\dt f(t)$ 
in $H^{2.5}$ and $f$ in $H^{3.5}$. We observe that $f(t)$ is expected to remain small in $H^{2.5}$ but has 
no reason to be small in $H^{3.5}$ (in particular because no smallness condition has been made on the norm 
of $f_0$ in $H^{3.5}$).

\section{The elliptic problem for the total pressure}

We first deduce from \eqref{mhd2} the elliptic system of equations solved by the total pressure. 
Applying $A^T \nabla \cdot$ to the equation for $v^\pm$ in \eqref{mhd2} gives
\begin{equation*}
-A^T \nabla \cdot (A^T\nabla \, Q^\pm) = A^T \nabla \cdot \Big\{ \dt v^\pm 
+(\tilde v^\pm \cdot \nabla) v^\pm -(\tilde B^\pm \cdot \nabla)B^\pm \Big\} \, .
\end{equation*}
Using the divergence relations $A^T\nabla\cdot v^\pm=A^T\nabla\cdot B^\pm=0$, we then deduce the equations
\begin{equation}
\label{equQ1}
-A^T \nabla \cdot (A^T\nabla \, Q^\pm) =\F^\pm \, ,
\end{equation}
where we have set
\begin{equation}
\label{defF}
\F^\pm := -\dt A_{ki} \, \dk v_i^\pm +A_{ki} \, \dk \tilde v^\pm \cdot \nabla v_i^\pm 
-\tilde v^\pm \cdot \nabla A_{ki} \, \dk v_i^\pm -A_{ki} \, \dk \tilde B^\pm \cdot \nabla B_i^\pm 
+\tilde B^\pm \cdot \nabla A_{ki} \, \dk B_i^\pm \, .
\end{equation}
Recalling that $a=J\, A$ we get from \eqref{equQ1} the equivalent equations
\begin{equation}
\label{equQ2}
-a^T \nabla \cdot (A^T \, \nabla Q^\pm) =J\, \F^\pm \, .
\end{equation}
Now we look for the boundary conditions satisfied by $Q^\pm$. Since $\tilde v_3^\pm=\tilde B^\pm_3=0$ and 
$\psi=v_3^\pm=B^\pm_3=0$ on $[0,T]\times\Gamma_{\pm}$, from the third equation for $v^\pm$ in \eqref{mhd2} 
evaluated on $\Gamma^\pm$ we obtain the homogeneous Neumann condition
\begin{equation}
\label{bcGpm}
\dtre Q^\pm=0 \qquad {\rm on} \; [0,T] \times \Gamma_{\pm} \, .
\end{equation}
On $\Gamma$ we take the scalar product of the equation for $v^\pm$ in \eqref{mhd2} with the vector $N$. We get
\begin{equation}
\label{bcG}
-(A^T \, \nabla Q^\pm) \cdot N = \Big\{ \dt v^\pm +(\tilde v^\pm \cdot \nabla) v^\pm 
-(\tilde B^\pm \cdot \nabla) B^\pm \Big\} \cdot N \, .
\end{equation}
Let us compute the jump of each quantity in \eqref{bcG} across $\Gamma$. Since $[Q]=0$ gives $[\duno Q] 
=[\ddue Q]=0$ on $[0,T] \times \Gamma$, we obtain (recall that $J=1$ on $\Gamma$)
\begin{equation}
\label{jumpQ1}
\big[ (A^T \, \nabla Q) \cdot N \big] = [A_{\ell j} \, N_j \, \partial_\ell Q] 
=(1+|\nabla'f|^2) \, [\dtre Q] \, .
\end{equation}
Using the boundary conditions $\dt f=v^\pm\cdot N, \; B^\pm\cdot N=0,$ on $[0,T]\times\Gamma$, we also deduce
\begin{equation}
\label{jumpQ2}
\left[ \{\dt v +(\tilde v \cdot \nabla) v -(\tilde B \cdot \nabla) B \} \cdot N \right] =\left[ 
2\, v' \cdot \nabla' \dt f +(v' \cdot \nabla') \nabla' f \cdot v' -(B' \cdot \nabla') \nabla' f 
\cdot B' \right] \, .
\end{equation}
Thus from \eqref{bcG}, \eqref{jumpQ1} and \eqref{jumpQ2}, we find the boundary condition
\begin{equation}
\label{jumpQ3}
[A_{\ell j} \, N_j \, \partial_\ell Q] =\G \qquad {\rm on} \; [0,T] \times \Gamma \, ,
\end{equation}
where we have set
\begin{equation}
\label{defG}
\G := -\left[ 2\, v' \cdot \nabla' \dt f +(v' \cdot \nabla') \nabla' f \cdot v' 
-(B' \cdot \nabla') \nabla' f \cdot B' \right] \, .
\end{equation}
Collecting the equations \eqref{equQ1}, \eqref{bcGpm}, \eqref{jumpQ3} gives the elliptic problem
\begin{equation}
\label{problemQ}
\begin{cases}
-A^T \nabla \cdot (A^T \, \nabla Q^\pm) =\F^\pm \, ,& {\rm on} \; [0,T] \times \Omega^\pm \, ,\\
[Q] =0 \, , & {\rm on} \; [0,T] \times \Gamma \, ,\\
[A_{\ell j} \, N_j \, \partial_\ell Q] =\G \, , & {\rm on} \; [0,T] \times \Gamma \, ,\\
\dtre Q^\pm =0 & {\rm on} \; [0,T] \times \Gamma_{\pm} \, ,\\
(x_1,x_2) \mapsto Q^\pm(t,x_1,x_2,x_3) & {\rm is\; 1-periodic,}
\end{cases}
\end{equation}
with $\F^\pm$ and $\G$ defined in \eqref{defF}, \eqref{defG}, respectively.

\begin{remark}
\label{remark2}
When one tries to solve the elliptic system for the pressure, it may be easier to work with the formulation 
\eqref{equQ2} instead of \eqref{equQ1} because of the necessary compatibility condition on the data $\F^\pm,\G$. 
More precisely, trying to solve problem \eqref{mhd2} by a fixed point argument, one possible step could be 
the resolution of system \eqref{problemQ}. (We have in mind the approach used in \cite{secchi93}, for the 
resolution of the incompressible MHD equations in a fixed domain under slip boundary conditions.) Thus the 
compatibilty condition needs to be satisfied by the data.

In order to formulate the compatibility condition we compute by an integration by parts
\begin{align*}
-\sum_\pm \int_{\Omega^\pm} a^T \nabla \cdot (A^T \, \nabla Q^\pm) \, {\rm d}x 
&= -\int_{\Gamma_+} a_{3i} \, A_{ki} \, \dk Q^+ \, {\rm d}x' 
+\int_{\Gamma} a_{3i} \, A_{ki} \, [\dk Q] \, {\rm d}x' 
+\int_{\Gamma_-} a_{3i} \, A_{ki} \, \dk Q^- \, {\rm d}x' \\
&+\sum_\pm \int_{\Omega^\pm} \dk a_{ki} \, A_{hi} \, \dh Q^\pm \, {\rm d}x \, ,
\end{align*}
where the last integral vanishes because of the so-called Piola's identity $\dk a_{ki} =0$. The boundary 
conditions for $Q$ yield
\begin{equation*}
-\sum_\pm \int_{\Omega^\pm} a^T \nabla \cdot (A^T \, \nabla Q^\pm) \, {\rm d}x 
=\int_\Gamma a_{3i} \, A_{ki} \, [\dk Q] \, {\rm d}x' =\int_\Gamma A_{ki} \, N_i \, [\dk Q] \, {\rm d}x' \, .
\end{equation*}
This shows that the data $\F,\G$ of problem \eqref{problemQ} need to satisfy the condition
\begin{equation*}
\sum_\pm \int_{\Omega^\pm} J \, \F^\pm \, {\rm d}x =\int_\Gamma \G \, {\rm d}x' \, .
\end{equation*}
This condition is satisfied with our definitions since
\begin{multline*}
\sum_\pm \int_{\Omega^\pm} J\, \F^\pm \, {\rm d}x =
\sum_\pm \int_{\Omega^\pm} a^T \nabla \cdot \{ \dt v^\pm 
+(\tilde v^\pm \cdot \nabla) v^\pm -(\tilde B^\pm \cdot \nabla) B^\pm \} \, {\rm d}x \\
=-\int_\Gamma \left[ N \cdot \{ \dt v +(\tilde v \cdot \nabla) v 
-(\tilde B \cdot \nabla) B \} \right] \, {\rm d}x' =\int_\Gamma \G \, {\rm d}x' \, ,
\end{multline*}
from \eqref{jumpQ2}, \eqref{defG}, and by computations as above. Thus the compatibility condition is satisfied.

Our approach here is different because we have already assumed that the solution exists and we only wish to 
prove an a priori estimate on a time interval that is independent of the solution. Consequently, we shall 
deal with the slightly more symmetric formulation \eqref{equQ1} to derive energy estimates.
\end{remark}

In the rest of this section we study the elliptic problem \eqref{problemQ} for generic data $\F^\pm,\G$. Only 
at the end of the section we will go back to the specific definition of $\F^\pm,\G$ given in \eqref{defF}, 
\eqref{defG}. As \eqref{problemQ} is time-independent, in the sense that time appears only as a parameter, 
for simplicity of notation from now on in this section the explicit dependence on $t$ will be neglected.

\subsection{The functional framework}

Thanks to the continuity of the total pressure across $\Gamma$, we can define the pressure $Q \in H^1(\Omega)$ 
by $Q :=Q^\pm$ on $\Omega^\pm$. The function $Q$ belongs to the Hilbert space
$$
{\mathcal V} :=\left\{ R \in H^1(\Omega) \, , \, \int_\Omega R\, {\rm d}x =0  \right\} \, .
$$
The space ${\mathcal V}$ equipped with the norm $\| \nabla R \|_{L^2(\Omega)}$ is indeed a Hilbert space, because of the 
Poincar\'e inequality, and the norm $\| \nabla R \|_{L^2(\Omega)}$ is equivalent to the standard $H^1$ norm. 
In what follows, the function $Q$ will be estimated in the space ${\mathcal V}$, and we shall repeatedly use 
the fact that the $L^2$ norm of $\nabla Q$ is equivalent to $\| Q^\pm \|_{1,\pm}$.

\subsection{The general procedure for the pressure estimate}

\underline{Step 1} We start from \eqref{problemQ}, multiply each equation in $\Omega^\pm$ by $Q^\pm$, integrate 
over $\Omega^\pm$ and use integration by parts. This yields
\begin{align*}
\sum_\pm \int_{\Omega^\pm} \partial_k (A_{kj} \, Q^\pm) \, A_{\ell j} \, \partial_\ell Q^\pm \, {\rm d}x 
&=\int_{\Gamma_+} A_{3j} \, Q^+ \, A_{\ell j} \, \partial_\ell Q^+ \, {\rm d}x' 
-\int_{\Gamma_-} A_{3j} \, Q^- \, A_{\ell j} \, \partial_\ell Q^- \, {\rm d}x' \\
&-\int_\Gamma A_{3j} \, Q^+ \, A_{\ell j} \, \partial_\ell Q^+ \, {\rm d}x' 
+\int_\Gamma A_{3j} \, Q^- \, A_{\ell j} \, \partial_\ell Q^- \, {\rm d}x' \\
&+\sum_\pm \int_{\Omega^\pm} Q^\pm \, \F^\pm \, {\rm d}x \, .
\end{align*}
We recall that from the boundary conditions, $\psi$ and $\dtre Q^\pm$ vanish on $\Gamma_\pm$ so the 
integrals on $\Gamma_\pm$ vanish. So 
we get
\begin{align*}
\sum_\pm \int_{\Omega^\pm} A_{kj} \, \partial_k Q^\pm \, A_{\ell j} \, \partial_\ell Q^\pm \, {\rm d}x = 
&-\sum_\pm \int_{\Omega^\pm} (\partial_k A_{kj}) \, Q^\pm \, A_{\ell j} \, \partial_\ell Q^\pm \, {\rm d}x \\
&-\int_\Gamma Q|_\Gamma \, \G \, {\rm d}x' +\sum_\pm \int_{\Omega^\pm} Q^\pm \, \F^\pm \, {\rm d}x \, ,
\end{align*}
where $Q|_\Gamma$ denotes the common trace of $Q^\pm$ on $\Gamma$. The integral on the left hand side 
gives the coercive term in $\nabla Q^\pm$ (see the definition \eqref{defAJa} and recall the condition 
$\| \nabla \psi \|_{L^\infty ([0,T'] \times \Omega)} \le 1/2$). Then we apply the Cauchy-Schwarz and 
Poincar\'e inequalities to derive
\begin{equation*}
c \, \| Q^\pm \|_{1,\pm}^2 \le \| \F^\pm \|_\pm^2 +\| \G \|_{H^{-0.5}(\T^2)}^2 
+\sum_\pm \int_{\Omega^\pm} |\partial_k A_{kj}| \, |Q^\pm| \, |\partial_\ell Q^\pm| \, {\rm d}x \, ,
\end{equation*}
for a suitable numerical constant $c>0$. Then we use the H\"older and Sobolev inequalities to derive
\begin{multline*}
\sum_\pm \int_{\Omega^\pm} |\partial_k A_{kj}| \, |Q^\pm| \, |\partial_\ell Q^\pm| \, {\rm d}x 
\le C \, \| \nabla Q^\pm \|_\pm \, |\nabla A|_4 \, |Q^\pm|_{4,\pm} \\
\le C \, \| A \|_2 \, \| Q^\pm \|_{1,\pm}^2 
\le C \, \| f(t) \|_{H^{2.5}(\T^2)} \, \| Q^\pm \|_{1,\pm}^2 \, .
\end{multline*}
Up to choosing $\eps_0$ small enough, we have thus derived the first estimate
\begin{equation}
\label{stimaQH1}
\forall \, t \in [0,T'] \, ,\quad 
\| Q^\pm \|_{1,\pm}^2 \le C_0 \, \left( \| \F^\pm \|_\pm^2 +\| \G \|_{H^{-0.5}(\T^2)}^2 \right) \, .
\end{equation}

\underline{Step 2} We are now going to estimate $Q^\pm$ in $H^2(\Omega^\pm)$. Let us first apply a tangential 
derivative $\dbar$ to \eqref{problemQ}, with $\dbar =\duno$ or $\dbar =\ddue$. Defining $\overline{Q}^\pm 
:=\dbar Q^\pm$, we obtain the elliptic system
\begin{equation}
\label{problemdunoQ}
\begin{cases}
-A^T \nabla \cdot (A^T \, \nabla \overline{Q}^\pm) = \overline{\F}^\pm \, ,& 
{\rm on} \; [0,T] \times \Omega^\pm \, ,\\
[\overline{Q}] =0 \, , & {\rm on} \; [0,T] \times \Gamma \, ,\\
[A_{\ell j} \, N_j \, \partial_\ell \overline{Q}] =\overline{\G} \, , & 
{\rm on} \; [0,T] \times \Gamma \, ,\\
\dtre \overline{Q}^\pm =0 & {\rm on} \; [0,T] \times \Gamma_{\pm} \, ,\\
(x_1,x_2) \mapsto \overline{Q}^\pm(t,x_1,x_2,x_3) & {\rm is\; 1-periodic,}
\end{cases}
\end{equation}
where the new source terms $\overline{\F}^\pm,\overline{\G}$ are defined by
\begin{align}
\overline{\F}^\pm &:= \dbar \F^\pm +\dbar A_{kj} \, \partial_k (A_{\ell j} \, \partial_\ell Q^\pm) 
+A_{kj} \, \partial_k ((\dbar A_{\ell j}) \, \partial_\ell Q^\pm) \, ,\label{defF1}\\
\overline{\G} &:=\dbar \G -\dbar (A_{\ell j} \, N_j) \, [\partial_\ell Q] 
=\dbar \G -\dbar (|\nabla'f|^2) \, [\dtre Q] \, .\label{defG1}
\end{align}
We apply the same procedure of integration by parts as above, obtaining first
\begin{align*}
\sum_\pm \int_{\Omega^\pm} A_{kj} \, \partial_k \overline{Q}^\pm \, A_{\ell j} \, 
\partial_\ell \overline{Q}^\pm \, {\rm d}x = 
&-\sum_\pm \int_{\Omega^\pm} (\partial_k A_{kj}) \, \overline{Q}^\pm \, A_{\ell j} \, 
\partial_\ell \overline{Q}^\pm \, {\rm d}x \\
&-\int_\Gamma \overline{Q}|_\Gamma \, \overline{\G} \, {\rm d}x' 
+\sum_\pm \int_{\Omega^\pm} \overline{Q}^\pm \, \overline{\F}^\pm \, {\rm d}x \, ,
\end{align*}
where $\overline{Q}|_\Gamma$ denotes the common trace of $\overline{Q}^\pm$ on $\Gamma$. The integrals 
on the left hand side give the coercive terms and, as above, we can absorb the first integrals occuring 
in the right hand side by choosing $\eps_0$ small enough. We thus have
\begin{equation*}
c \, \| \overline{Q}^\pm \|_{1,\pm}^2 \le 
-\int_\Gamma \overline{Q}|_\Gamma \, \overline{\G} \, {\rm d}x' 
+\sum_\pm \int_{\Omega^\pm} \overline{Q}^\pm \, \overline{\F}^\pm \, {\rm d}x \, .
\end{equation*}
We now estimate the integrals on $\Omega^\pm$, recalling the definition \eqref{defF1} for $\overline{\F}^\pm$. 
Let us first observe that the term with $\dbar \F^\pm$ can be integrated by parts and we can then apply 
Cauchy-Schwarz and Young inequalities. The other terms are estimated as follows:
\begin{multline*}
\sum_\pm \int_{\Omega^\pm} |\overline{Q}^\pm| \, |\dbar A_{kj}| \, |A_{\ell j}| \, |\partial_k \partial_\ell Q^\pm| 
\, {\rm d}x \le C \, \| Q^\pm \|_{2,\pm} \, |\nabla A|_4 \, | A|_\infty \, |\overline{Q}^\pm|_{4,\pm} \\
\le C \, \| A \|_2^2 \, \| Q^\pm \|_{2,\pm}^2 
\le C \, \| f(t) \|_{H^{2.5}(\T^2)}^2 \, \| Q^\pm \|_{2,\pm}^2 \, ,
\end{multline*}
\begin{multline*}
\sum_\pm \int_{\Omega^\pm} |\overline{Q}^\pm| \, |\dbar A_{kj}| \, |\partial_k A_{\ell j}| \, |\partial_\ell Q^\pm| 
\, {\rm d}x \le C \, |\overline{Q}^\pm|_{4,\pm} \, |\nabla A|_4^2 \, |\nabla Q^\pm|_{4,\pm} \\
\le C \, \| A \|_2^2 \, \| Q^\pm \|_{2,\pm}^2 
\le C \,  \| f(t) \|^2_{H^{2.5}(\T^2)} \, \| Q^\pm \|_{2,\pm}^2 \, ,
\end{multline*}
and applying similar sequences of inequalities, the reader can get quickly convinced that all other terms 
in the product $\overline{Q}^\pm \, \overline{\F}^\pm$ are estimated by the same quantity. We thus have
\begin{equation*}
c \, \| \overline{Q}^\pm \|_{1,\pm}^2 \le \| \F^\pm \|_\pm^2 
+\left| \int_\Gamma \overline{Q}|_\Gamma \, \overline{\G} \, {\rm d}x' \right| + C \,  \| f(t) \|^2_{H^{2.5}(\T^2)} \, \| Q^\pm \|_{2,\pm}^2\, .
\end{equation*}
Let us now turn to the boundary term. Of course, we have
\begin{equation*}
\left| \int_\Gamma \overline{Q}|_\Gamma \, \dbar \G \, {\rm d}x' \right| \le \| \G \|_{H^{0.5}(\T^2)} \, 
\| \overline{Q}|_\Gamma \|_{H^{0.5}(\Gamma)} \le C \, \| \G \|_{H^{0.5}(\T^2)} \, \| \overline{Q}^\pm \|_{1,\pm} \, .
\end{equation*}
The remaining term occuring in $\overline{\G}$ is easily estimated as follows:
\begin{multline*}
\left| \int_\Gamma \overline{Q}|_\Gamma \, [\dtre Q] \, \dbar (|\nabla'f|^2) \, {\rm d}x' \right| 
\le \left| \overline{Q}|_\Gamma \right|_3 \, |[\dtre Q]|_3 \, \left| \dbar (|\nabla'f|^2) \right|_3 \\
\le C \, \| \overline{Q}|_\Gamma \|_{H^{0.5}(\Gamma)} \, \| [\dtre Q] \|_{H^{0.5}(\Gamma)} \, 
\| |\nabla'f|^2 \|_{H^{1.5}(\T^2)} 
\le C \,  \| Q^\pm \|^2_{2,\pm} \, \| f(t) \|_{H^{2.5}(\T^2)}^2 \, ,
\end{multline*}
where we have used $H^{0.5}(\Gamma) \hookrightarrow L^4(\Gamma)$ (which holds in two space dimensions), and the 
fact that $H^{1.5}(\Gamma)$ is an algebra. Applying Young's inequality again, we thus obtain
\begin{equation}
\label{stimadbarQ}
\| \overline{Q}^\pm \|_{1,\pm}^2 \le C_0 \, \left( \| \F^\pm \|_\pm^2 +\| \G \|_{H^{0.5}(\T^2)}^2 
+\| f(t) \|^2_{H^{2.5}(\T^2)} \, \| Q^\pm \|_{2,\pm}^2 \right) \, .
\end{equation}

\underline{Step 3} The remaining second order derivative $\dtre^2 Q^\pm$ is estimated directly from the 
equation \eqref{problemQ} by using the explicit expression of the coefficients $A_{kj}$. More precisely, 
\eqref{problemQ} reads
\begin{equation*}
A_{ji} \, A_{ki} \, \partial_j \partial_k Q^\pm =-\F^\pm -A_{ji} \, \partial_j A_{ki} \, \partial_k Q^\pm \, ,
\end{equation*}
that is,
\begin{equation}
\label{equationdtreQ}
\dfrac{1+|\nabla' \psi|^2}{(1+\dtre \psi)^2} \, \dtre^2 Q^\pm +\duno^2 Q^\pm +\ddue^2 Q^\pm 
-2\dfrac{\duno \psi \, \duno \dtre Q^\pm}{1+\dtre \psi} -2\dfrac{\ddue \psi \, \ddue \dtre Q^\pm}{1+\dtre \psi} 
=-\F^\pm -A_{ji} \, \partial_j A_{ki} \, \partial_k Q^\pm \, .
\end{equation}
We thus obtain
\begin{multline*}
c \, \| \dtre^2 Q^\pm \|_\pm^2 \le C \, \left( \| \overline{Q}^\pm \|_{1,\pm}^2 +\| \F^\pm \|_\pm^2 
+\| A \, \partial^1 A \, \partial^1 Q^\pm \|_\pm^2 \right) \\
\le C \, \left( \| \overline{Q}^\pm \|_{1,\pm}^2 +\| \F^\pm \|_\pm^2 
+\| f(t) \|^2_{H^{2.5}(\T^2)} \, \| Q^\pm \|_{2,\pm}^2 \right) \, .
\end{multline*}
Combining with \eqref{stimaQH1} and \eqref{stimadbarQ} and choosing the numerical constant $\eps_0$ 
sufficiently small, we obtain
\begin{equation}
\label{stimaQH2}
\forall \, t \in [0,T'] \, ,\quad 
\| Q^\pm \|_{2,\pm}^2 \le C_0 \, \left( \| \F^\pm \|_\pm^2 +\| \G \|_{H^{0.5}(\T^2)}^2 \right) \, .
\end{equation}

\underline{Step 4} We now apply the estimate \eqref{stimaQH2} to the solution $\overline{Q}^\pm$ to the problem 
\eqref{problemdunoQ}, which has the same form as \eqref{problemQ} but with different source terms (defined in 
\eqref{defF1} and \eqref{defG1}). We thus have
\begin{equation*}
\forall \, t \in [0,T'] \, ,\quad 
\| \overline{Q}^\pm \|_{2,\pm}^2 \le C \, \left( \| \overline{\F}^\pm \|_\pm^2 
+\| \overline{\G} \|_{H^{0.5}(\T^2)}^2 \right) \, .
\end{equation*}
The $L^2$-estimate of $\overline{\F}^\pm$ follows by applying similar arguments as above; for instance, we have
\begin{equation*}
\|\partial^1 A \, \partial^1 A \, \partial^1 Q^+\|_+ \le \|\partial^1 A \, \partial^1 A\|_+ \, 
\| Q^+ \|_{W^{1,\infty}(\Omega^+)} \le C \, |\nabla A|_4^2 \, \| Q^+ \|_{3,+} \le 
C \, \| f(t) \|_{H^{2.5}(\T^2)}^2 \, \| Q^\pm \|_{3,\pm} \, .
\end{equation*}
All the other terms in $\overline{\F}^\pm$ admit the same upper bound, that is
\begin{equation*}
\| \overline{\F}^\pm \|_\pm^2 \le C \, \left( \| \F^\pm \|_{1,\pm}^2 
+C \, \| f(t) \|_{H^{2.5}(\T^2)}^2 \, \| Q^\pm \|_{3,\pm}^2 \right) \, .
\end{equation*}
As far as the boundary source term is concerned, we apply Lemma \ref{lemma4} and obtain
\begin{equation*}
\| \dbar (|\nabla'f|^2) \, [\dtre Q] \|_{H^{0.5}(\Gamma)} \le C \, \| \dbar (|\nabla'f|^2) \|_{H^{0.5}(\T^2)} 
\, \| [\dtre Q] \|_{H^{1.5}(\Gamma)} \le C \, \| f(t) \|_{H^{2.5}(\T^2)}^2 \, \| Q^\pm \|_{3,\pm}^2 \, .
\end{equation*}
We have thus derived the upper bound
\begin{equation*}
\forall \, t \in [0,T'] \, ,\quad 
\| \overline{Q}^\pm \|_{2,\pm}^2 \le C \, \left( \| {\F}^\pm \|_{1,\pm}^2 
+\| {\G} \|_{H^{1.5}(\T^2)}^2 +\| f(t) \|^2_{H^{2.5}(\T^2)} \, \| Q^\pm \|_{3,\pm}^2 \right) \, .
\end{equation*}
The remaining third order derivative $\dtre^3 Q^\pm$ can be estimated by applying $\dtre$ to the equation 
\eqref{equationdtreQ}. The commutators are estimated exactly as above, and we now feel free to skip a few 
details. Eventually, up to choosing a sufficiently small numerical constant $\eps_0>0$, and provided that 
$T'$ is such that \eqref{uniforma} holds, we derive the estimate
\begin{equation}
\label{stimaQH3}
\forall \, t \in [0,T'] \, ,\quad 
\| Q^\pm \|_{3,\pm}^2 \le C_0 \, \left( \| \F^\pm \|_{1,\pm}^2 +\| \G \|_{H^{1.5}(\T^2)}^2 \right) \, .
\end{equation}

\subsection{The final pressure estimate}

It only remains to use the definition of the source terms $\F^\pm,\G$ in \eqref{stimaQH3}. Using first the 
fact that $H^{1.5}(\T^2)$ is an algebra and recalling the definition \eqref{defG} of $\G$, we have
\begin{equation*}
\| \G(t) \|_{H^{1.5}(\T^2)} \le C \, \left( 
 \| v^\pm(t) \|_{3,\pm} \| \dt f(t) \|_{H^{2.5}(\T^2)} + \| v^\pm(t),B^\pm(t) \|^2_{3,\pm} \| f(t) \|_{H^{3.5}(\T^2)} \right) \, ,
\end{equation*}
and using \eqref{stima-vB}, \eqref{stimaf4}, we get
\begin{equation*}
\| \G(t) \|^2_{H^{1.5}(\T^2)} \le M_0 +t \, F(\max_{0 \le s \le t} {\mathcal E}(s)) \, .
\end{equation*}

The source terms $\F^\pm$ can be estimated by applying the classical estimate
\begin{equation*}
\| u_1 \, u_2 \|_{H^1} \le C \, 
(\| u_1 \|_{L^\infty} \, \| u_2 \|_{H^1} +\| u_2 \|_{L^\infty} \, \| u_1 \|_{H^1}) \, .
\end{equation*}
Analyzing each separate term in the definition \eqref{defF} of $\F^\pm$ by applying the latter product 
estimate and by using \eqref{uniform}, \eqref{stima-vB} or \eqref{stimaf4}, we get
\begin{equation*}
\| \F^\pm(t) \|_{1,\pm}^2 \le M_0 +t\, F(\max_{0 \le s \le t} {\mathcal E}(s)) \, .
\end{equation*}
Adding the previous two inequalities, we obtain our final estimate for the pressure:
\begin{equation}
\label{stimaQ}
\forall \, t \in [0,T'] \, ,\quad 
\| Q^\pm \|_{3,\pm}^2 \le M_0 +t\, F(\max_{0 \le s \le t} {\mathcal E}(s)) \, .
\end{equation}

\section{Proof of Theorem \ref{mainthm}}

If we summarize the analysis of the previous sections, we have shown that there exist some numerical 
constants $\eps_0>0$ and $M_0>0$, there exists a nonnegative nondecreasing function $F$ on $\R^+$, all 
three depending only on $\delta_0$ and $R$ such that, on any time interval $[0,T']$ for which the 
inequalities \eqref{uniform} are valid, there holds
\begin{equation}
\label{relationE}
\forall \, t \in [0,T'] \, ,\quad {\mathcal E}(t) \le M_0 +t\, F(\max_{0 \le s \le t} {\mathcal E}(s)) \, .
\end{equation}
The function $F$ and the constants $\eps_0,M_0$ are independent of the particular solution that we are 
considering. Moreover, $H^2(\Omega^\pm)$ is an algebra so applying direct estimates on \eqref{mhd2} we 
find
\begin{equation*}
\forall \, t \in [0,T'] \, ,\quad \| \dt v^\pm(t),\dt B^\pm(t) \|_{2,\pm} \le F({\mathcal E}(t)) \, ,
\end{equation*}
so integrating with respect to $t$ we have
\begin{equation}
\label{conditionE}
\forall \, t \in [0,T'] \, ,\quad \|  v^\pm(t)-v^\pm_0, B^\pm(t)-B^\pm_0 \|_{2,\pm} 
\le t\, F(\max_{0 \le s \le t} {\mathcal E}(s)) \, .
\end{equation}

From now on, the nonnegative nondecreasing function $F$ is fixed, as well as the constants $\eps_0$, $M_0$. 
To complete the proof of Theorem \ref{mainthm}, we define $\eps_1 := \eps_0/2$, and we choose a time $T_0>0$ 
such that $2\, T_0 \, F(2\, M_0) \le M_0$ and $2\, T_0 \, F(2\, M_0) \le \eps_1$. We emphasize that the definition of $T_0$ only depends on $\delta_0$ and 
$R$.
Then we define $T'$ as the 
maximal time on which \eqref{uniform} holds ($T'$ is positive because \eqref{uniform} holds at the initial 
time with a strict inequality). We will see that $T_0\le T'$ if $T_0<T$, and $T'= T< T_0$ if $T<T_0.$

There are now two possibilities. Let us first assume $T>T_0$, and let us define $I$ as the set of all times 
$t \in [0,T_0]$ such that
\begin{equation*}
\max_{0 \le s \le t} {\mathcal E}(s) \le 2\, M_0 \, ,\quad 
\max_{0 \le s \le t} \|  v^\pm(s)-v^\pm_0, B^\pm(s)-B^\pm_0 \|_{2,\pm} \le \eps_0 \, ,\quad 
\max_{0 \le s \le t} \| f(s) \|_{H^{2.5}(\T^2)} \le \eps_0 \, .
\end{equation*}
Then $I$ is non-empty since it contains $0$ (use \eqref{relationE} for $t=0$), and $I$ is closed since all 
functions involved in the definition of $I$ are continuous. Let us show that $I$ is open. Let $\underline{t} 
\in I$. Using \eqref{relationE}, we have
\begin{equation*}
{\mathcal E}(\underline{t}) \le M_0 +\underline{t} \, F(\max_{0 \le s \le \underline{t}} {\mathcal E}(s)) 
\le M_0 +T_0 \, F(2\, M_0) <2\, M_0 \, .
\end{equation*}
In the same way, \eqref{stimaf4}, \eqref{conditionE} and the definition of $\eps_1$ give
\begin{equation*}
\|  v^\pm(\underline{t})-v^\pm_0, B^\pm(\underline{t})-B^\pm_0 \|_{2,\pm} <\eps_0 \, ,\quad 
\| f(\underline{t}) \|_{H^{2.5}(\T^2)} <\eps_0 \, .
\end{equation*}
Consequently, there exists a neighborhood of $\underline{t}$ in $[0,T_0]$ that is included in $I$. In 
other words, $I$ is open. Hence $I=[0,T_0]$ and the result of Theorem \ref{mainthm} is proved. The proof 
in the case $T \le T_0$ is similar.

\section{Proof of Lemma \ref{lemma1}}
\label{prooflemma1}

Given $\chi\in C^\infty_0(\R)$, $\chi=1$ on $[-1,1]$, we define
\begin{equation}
\label{def-f1}
f^{(1)}(x',x_3) :=\chi (x_3|D|) \, f(x') \, ,\quad \psi (x',x_3) :=(1-x_3^2) \, f^{(1)}(x',x_3) \, ,
\end{equation}
where $\chi(x_3|D|)$ is the pseudo-differential operator with $|D|$ being the Fourier multiplier in the 
variables $x'$. From the definition it readily follows that $\psi(x',0)=f(x')$, $\psi(x',\pm 1)=0$ for 
all $x'\in\T^2$. Moreover, 
\begin{equation}
\label{d3psi}
\dtre \psi(x',x_3) =-2\, x_3 \, f^{(1)}(x',x_3) +(1-x_3^2) \, \chi'(x_3|D|) \, |D| \, f(x') \, ,
\end{equation}
which vanishes if $x_3=0$. Given any function $g$ defined on $\T^2$, let us denote by $c_k(g)$ the $k$-th 
Fourier coefficient
$$
c_k(g) = \int_{\T^2} {\rm e}^{-2\, i \, \pi \, k\cdot x'} \, g(x') \, {\rm d}x' \, ,\quad k \in \Z^2 \, .
$$
Since $c_k(f^{(1)}(\cdot,x_3))=\chi(x_3\, |k|) \, c_k(f)$, we compute
\begin{multline*}
\|\psi(\cdot,x_3)\|^2_{H^m(\T^2)} =(1-x_3^2)^2 \, \| f^{(1)}(\cdot,x_3) \|^2_{H^m(\T^2)} 
\le C \, (1-x_3^2)^2 \, \sum_{k \in \Z^2} (1+|k|^2)^m \, \left| c_k(f^{(1)}(\cdot,x_3)) \right|^2 \\
\le C \, (1-x_3^2)^2 \, \sum_{k \in \Z^2} (1+|k|^2)^m \, \chi^2(x_3 \, |k|) \, |c_k(f)|^2 \, .
\end{multline*}
It follows that
\begin{align*}
\| \psi \|^2_{L^2_{x_3}(H^m(\T^2))} &\le C\, \int^1_{-1} (1-x_3^2)^2 \, 
\sum_{k \in \Z^2} (1+|k|^2)^m \, \chi^2(x_3\, |k|) \, |c_k(f)|^2 \, {\rm d}x_3 \\
&\le C\, \sum_{k \in \Z^2} (1+|k|^2)^m \, |c_k(f)|^2 \, \int^1_{-1} \chi^2 (x_3\, |k|) \, {\rm d}x_3 \\
&\le C\, |c_0(f)|^2 +C\, \sum_{|k| \ge 1} (1+|k|^2)^m \, |c_k(f)|^2 \, \dfrac{1}{|k|} \, 
\int^{|k|}_{-|k|} \chi^2(s) \, {\rm d}s \, .
\end{align*}
Denoting by $X\in C^\infty(\R)$ the primitive function of $\chi^2$ vanishing at $-\infty$, i.e. $X'(s)=\chi^2(s)$, 
we notice that $X$ is bounded over all $\R$. Then
\begin{equation}
\label{L2psi}
\| \psi \|^2_{L^2_{x_3}(H^m(\T^2))} \le C\, |c_0(f)|^2 +C\, \sum_{|k| \ge 1} (1+|k|^2)^{m-1/2} 
|c_k(f)|^2 \, \sup_{s \in \R} |X(s)| \, \le C\, \| f \|_{H^{m-1/2}(\T^2)}^2 \, .
\end{equation}
In a similar way, from \eqref{d3psi}, we obtain
\begin{align*}
\| \dtre \psi \|^2_{L^2_{x_3}(H^{m-1}(\T^2))} &\le C \, \Big( 
\| \chi(x_3\, |D|) \, f \|^2_{L^2_{x_3}(H^{m-1}(\T^2))} 
+\| \chi'(x_3\, |D|) \, |D| \, f \|^2_{L^2_{x_3}(H^{m-1}(\T^2))} \Big) \\
&\le C\, \sum_{k \in \Z^2} (1+|k|^2)^{m-1} \, |c_k(f)|^2 \, \int^1_{-1} \chi^2 (x_3\, |k|) \, {\rm d}x_3 \\
&+C\, \sum_{k \in \Z^2} (1+|k|^2)^{m-1} \, |k|^2 \, |c_k(f)|^2 \, 
\int^1_{-1} |\chi'(x_3 \, |k|)|^2 \, {\rm d}x_3 \\
&\le C\, \| f \|_{H^{m-3/2}(\T^2)}^2 +C\, \sum_{k \neq 0} (1+|k|^2)^{m-1} |k| \, |c_k(f)|^2 \, 
\int^{|k|}_{-|k|} |\chi'(s)|^2 \, {\rm d}s \, .
\end{align*}
Denoting by $Y\in C^\infty(\R)$ a primitive function of $(\chi')^2$, we also notice that $Y$ is bounded 
over all $\R$, so as in \eqref{L2psi}, we get
\begin{equation*}
\| \dtre \psi \|^2_{L^2_{x_3}(H^{m-1}(\T^2))} \le C\, \| f \|_{H^{m-3/2}(\T^2)}^2 
+C\, \sum_{|k| \ge 1} (1+|k|^2)^{m-1/2} \, |c_k(f)|^2 \, \sup_{s\in\R} |Y(s)| 
\le C\, \|f\|_{H^{m-1/2}(\T^2)}^2 \, .
\end{equation*}
Iterating the same argument yields
\begin{equation*}
\| \dtre^j \psi \|^2_{L^2_{x_3}(H^{m-j}(\T^2))} \le C\, \| f \|_{H^{m-1/2}(\T^2)}^2 \, ,\quad j=0,\dots,m \, .
\end{equation*}
Adding over $j=0,\dots, m$ finally gives $\psi \in H^m(\Omega)$ and the continuity of the map $f \mapsto \psi$.

The proof of Lemma \ref{lemma2} follows from Lemma \ref{lemma1}, with $t$ as a parameter. Notice also that 
the map $f\to f^{(1)}$, see \eqref{def-f1}, is linear and that the time regularity is conserved because, with 
obvious notation, $(\dt^j f)^{(1)} =\dt^j (f^{(1)})$. The conclusions of Lemma \ref{lemma2} follow directly.

\bibliographystyle{plain}
\bibliography{CMST}
\end{document}